\documentclass[a4]{article}   	
\usepackage{hyperref}
\usepackage{amsmath,amssymb,amsthm,bbm}
\usepackage{tikz}
\usepackage{todonotes}
\usetikzlibrary{arrows,decorations.pathmorphing}
\tikzset{snake it/.style={decorate, decoration={snake, amplitude=.3mm,segment length=4mm}}}
\usepackage{color}
\usepackage[notref,notcite]{showkeys}
\newcommand{\Z}{\mathbb{Z}}
\newcommand{\mc}[1]{\mathcal{#1}}

\renewcommand{\P}{\mathbb{P}}
\newcommand{\E}{\mathbb{E}}
\newcommand{\R}{\mathbb{R}}
\newcommand{\N}{\mathbb{N}} 
\newcommand{\Q}{\mathbb{Q}}

\newcommand{\sss}[1]{\scriptscriptstyle{#1}}
\newcommand{\blank}[1]{}

\newcommand{\vv}{\boldsymbol{1}}

\newcommand{\Pp}{\mathbb{P}_p}
\newcommand{\Ep}{\mathbb{E}_p}

\title{A shape theorem for the orthant model} 
\author{Mark Holmes \& Thomas S. Salisbury}

\newtheorem{innercustomthm}{Theorem}
\newenvironment{customthm}[1]
{\renewcommand\theinnercustomthm{#1}\innercustomthm}
{\endinnercustomthm}

\newcounter{thmcounter}
\newtheorem{theorem}[thmcounter]{Theorem}

\newcounter{opencounter}
\newtheorem{open}[opencounter]{Open problem}
\newcounter{lemcounter}
\newtheorem{lemma}[lemcounter]{Lemma}
\newcounter{corcounter}
\newtheorem{corollary}[corcounter]{Corollary}

\newtheorem*{theorem*}{Theorem}

\begin{document}
\maketitle

\begin{abstract}
	We study a particular model of a random medium, called the orthant model, in general dimensions $d\ge 2$.
Each site $x\in \Z^d$ independently has arrows pointing to its positive neighbours $x+e_i$, $i=1,\dots, d$ with probability $p$ and otherwise to its negative neighbours $x-e_i$, $i=1,\dots, d$ (with probability $1-p$).
We prove a shape theorem for the set of sites reachable by following arrows, starting from the origin, when $p$ is large.  The argument uses subadditivity, as would be expected from the shape theorems arising in the study of first passage percolation. The main difficulty to overcome is that the primary objects of study are not stationary, which is a key requirement of the subadditive ergodic theorem.
\end{abstract}

\section{The model and main results}
Fix $d\ge 2$, and set $[d]=\{1,2,\dots, d\}$.   Let $\mc{E}_+=\{e_i\}_{i \in [d]}$ 
denote the set of canonical basis vectors for $\Z^d$ and let $\mc{E}_-=\{-e_i\}_{i \in [d]}$ and $\mc{E}=\mc{E}_+\cup \mc{E}_-$.  Let $o$ denote the origin in $\Z^d$.  Let $\mu$ be a probability measure on the power set of $\mc{E}$.  

Let $(\mc{G}_x)_{x \in \Z^d}$ be i.i.d.~with law $\mu$.  This induces a random directed graph on $\Z^d$ - insert arrows from $x$ to each of the vertices $\{x+e:e\in \mc{G}_x\}$.  These kinds of models are called {\em degenerate random environments} \cite{DRE,DRE2}, and their study is motivated by the fact that they lay the foundation for understanding the behaviour of {\em random walks in non-elliptic random environments} (see  \cite{RWDRE,RWDRE2} for the non-elliptic setting and \cite{Zeit04} for the general theory in the uniformly elliptic setting).  We are interested in the set of vertices $\mc{C}_x\subset \Z^d$ that can be reached from $x$ by following these arrows, as well as the sets $\mc{B}_x=\{y \in \Z^d:x \in \mc{C}_y\}$ and $\mc{M}_x=\mc{C}_x\cap \mc{B}_x$.  

Our goal is to obtain a shape theorem for $\mc{C}_o$, under a particular choice of $\mu$ (to be described below). See Theorem \ref{shapetheorem}. As one would expect, this is proved using the subadditive ergodic theorem. The key technical problem for us to overcome is that the basic object of study (labelled $\beta_n(u)$ below) fails to have the stationarity properties required by the subadditive ergodic theorem, because of the special role of the origin $o$. Our approach is to find a substitute quantity that does have the required stationarity, and then bootstrap our way from that to $\beta_n(u)$ using geometric arguments and estimates based on large deviations and the enumeration of self-avoiding walks. We hope that this approach will be useful in other situations, where the subadditive ergodic theorem does not directly apply. We expand on the relationship between our arguments and earlier work, at the end of Section \ref{sec:subadditive}.

The \emph{orthant model} is an elegant example of this class of models in which $\mu(\{\mc{E}_+\})=p=1-\mu(\{\mc{E}_-\})$.  When $d=2$ this model is dual to oriented site percolation (OTSP) on the triangular lattice \cite{DRE,DRE2}.   This fact has been exploited to prove a shape theorem for $\mc{C}_o$ in 2 dimensions, as well as to give improved estimates of the critical parameter for OTSP.   
The left side of Figure \ref{fig:2d_e1e2} shows an example of $\mc{C}_o$ when $d=2$.  Figure \ref{fig:3d} shows an example of $\mc{C}_o$ when $d=3$.
\begin{figure}
\hspace{-1cm}
\includegraphics[scale=.4]{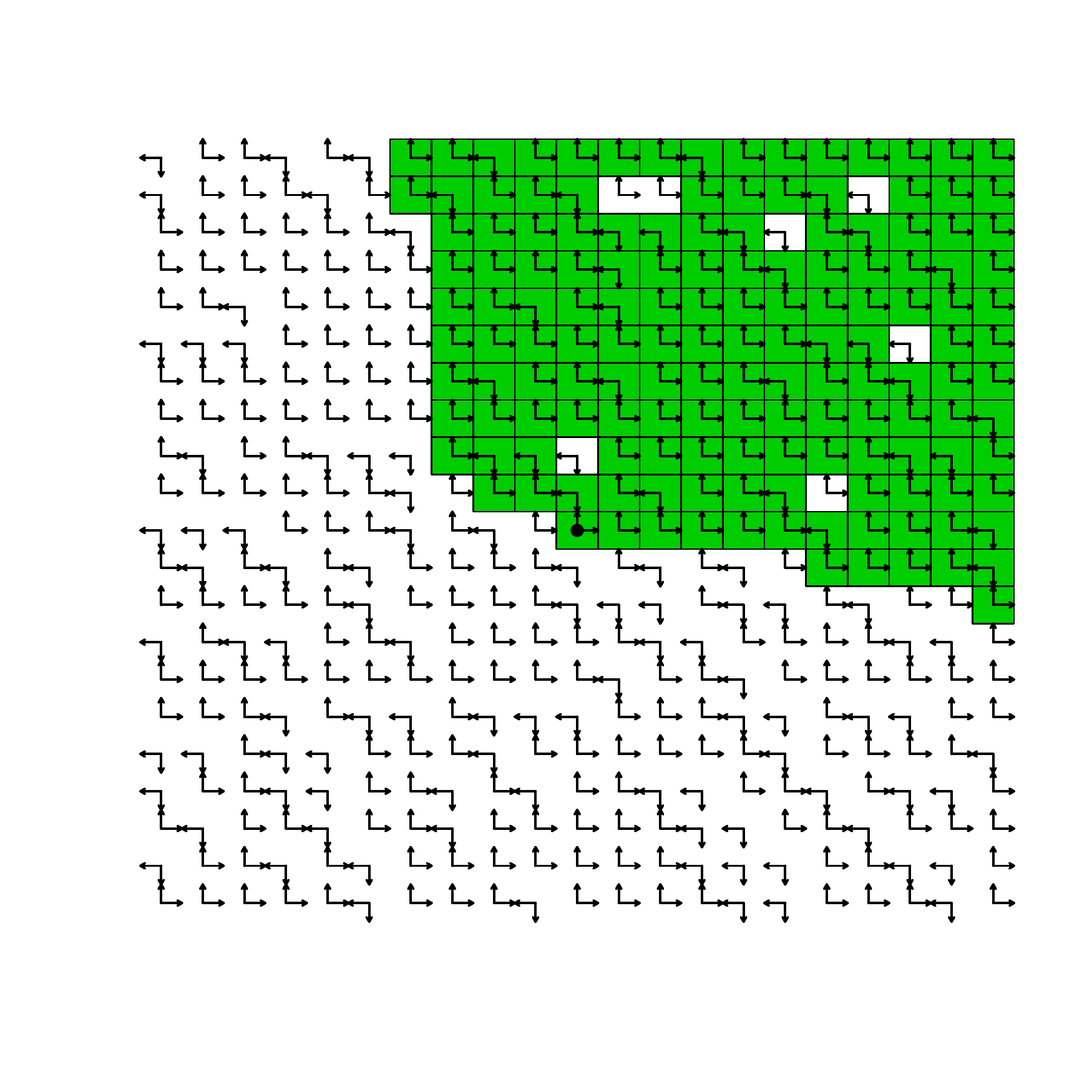}
\includegraphics[scale=.4]{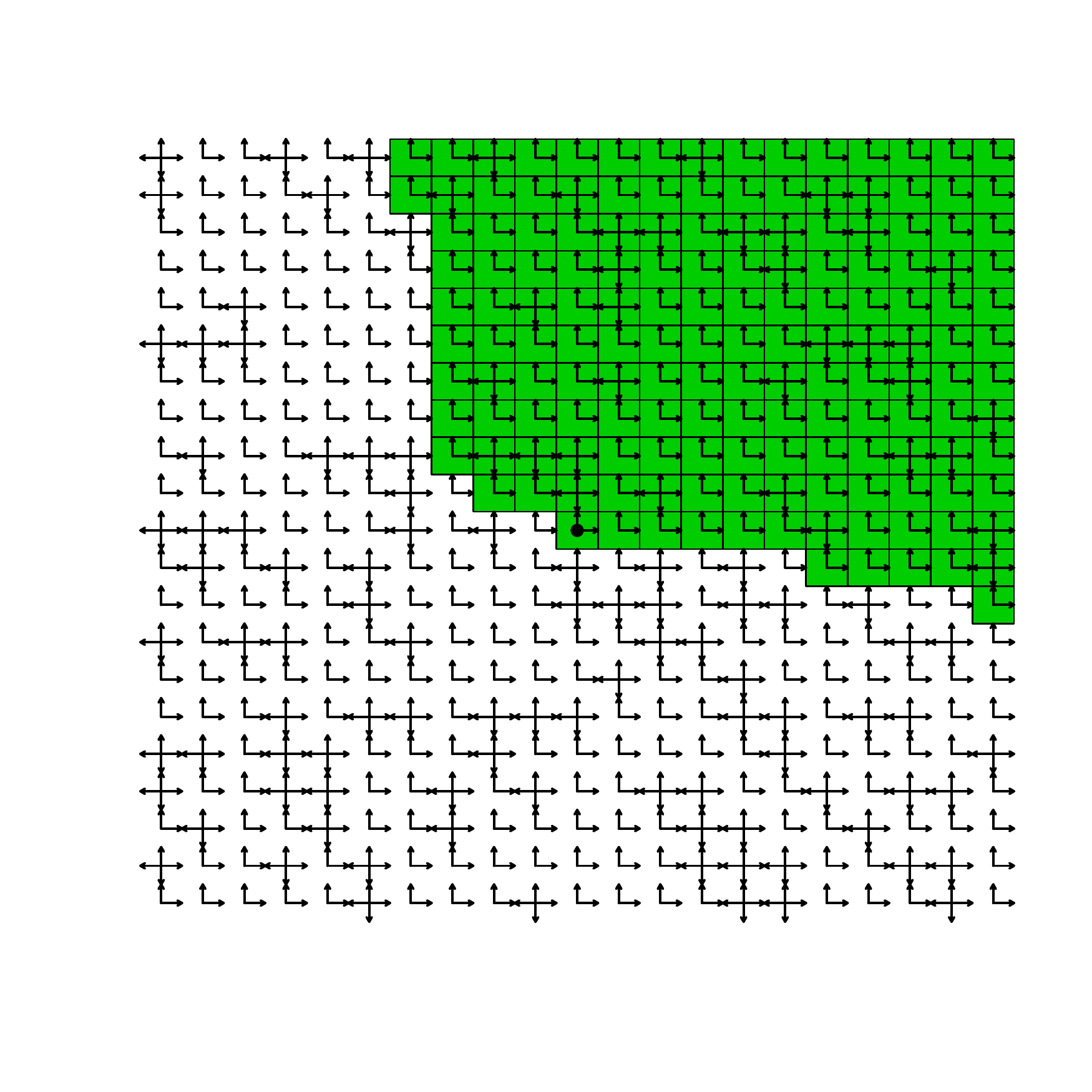}
\vspace{-1.5cm}
\caption{Coupled realisations of finite parts of the set $\mc{C}_o$ for the orthant and half-orthant models with $p=0.7$ and $d=2$. Note that the boundaries of the two shaded clusters are the same (see Theorem \ref{thm:other1}).}
\label{fig:2d_e1e2}
\end{figure}

\begin{figure}
	\begin{tabular}{lcr}
\hspace{-2cm}
\includegraphics[scale=1]{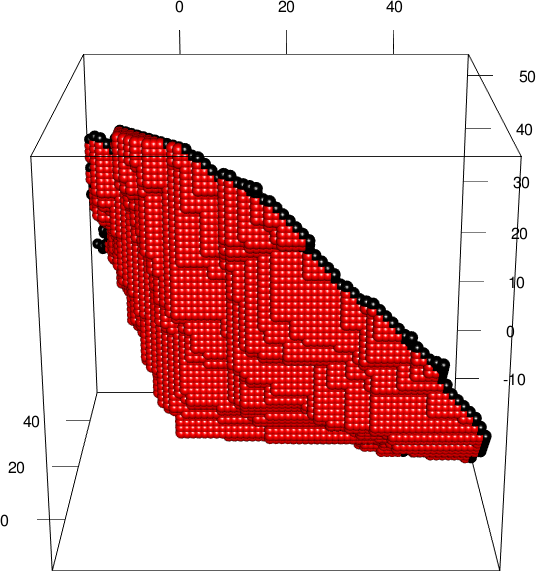} &
\includegraphics[scale=1]{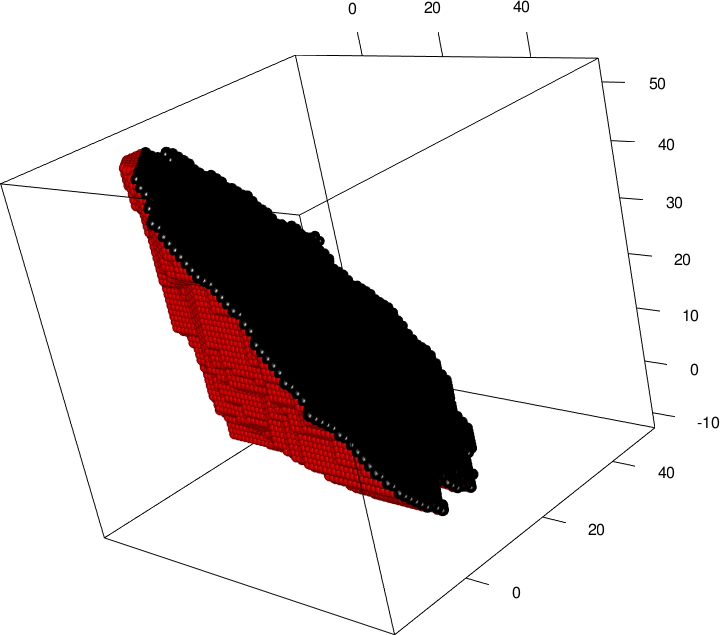} &
\includegraphics[scale=1]{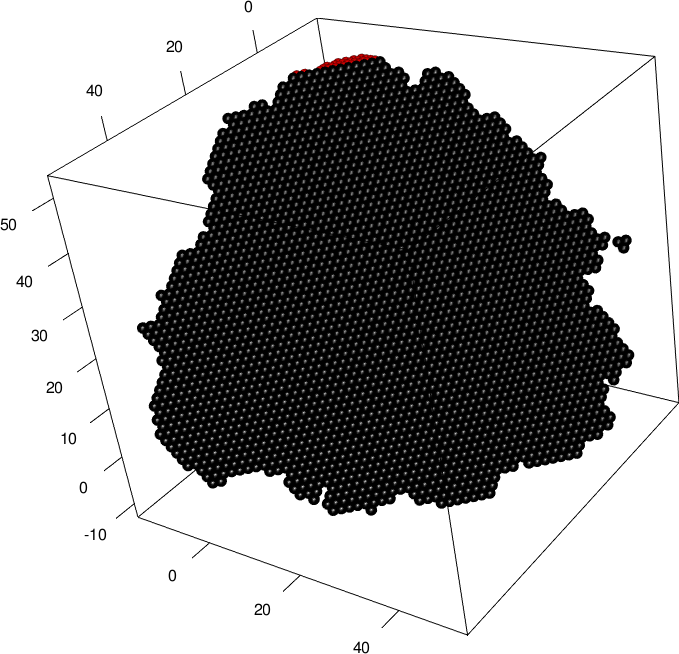}
	\end{tabular}
	\caption{A simulation of part of the cluster $\mc{C}_o(.95)$ for the orthant model, viewed from 3 different angles.  In each case the black/dark vertices are a cross-section where the sum of coordinates is equal to $50$.}
	\label{fig:3d}
\end{figure}
 Let $\Omega_+=\{x\in \Z^d:\mc{G}_x=\mc{E}_+\}$ and $\Omega_-=\Z^d\setminus \Omega_+$.  The orthant model has the property that $\mc{C}_o$ is almost surely infinite (this means that a random walk on the random directed graph $\mc{G}=(\mc{G}_x)_{x \in \Z^d}$ will visit infinitely many sites), since e.g.~it contains an infinite path consisting entirely of $e_1$ steps from $\Omega_+$ sites and $-e_2$ steps from $\Omega_-$ sites.  It also has the property that $\mc{C}_o$ is non-monotone in $p$ under the standard coupling of environments (as in site percolation, see e.g.~\eqref{coupling} below).  Nevertheless it has recently been proved \cite{DREphase} that for each $d\ge 2$ there is a phase transition in the structure of $\mc{C}_o$ as $p$ varies.  In order to state this result, we define the {\em half-orthant model} to be the model with $\mu(\{\mc{E}_+\})=p=1-\mu(\{\mc{E}\})$.  The orthant and half-orthant models can be defined on the same probability space such that $\mc{C}_o(p)\subset \mc{C}^*_o(p)$ for every $p\in [0,1]$ (where the asterisk refers to the half-orthant model) as follows:  Let $(U_x)_{x \in \Z^d}$ be i.i.d.~standard uniform random variables and set
 \begin{equation}
 \mc{G}_x(p)=\mc{E}_+  \iff \mc{G}^*_x(p)=\mc{E}_+ \iff U_x\le p.\label{coupling}
 \end{equation}
For $x\in \Z^{d}$ let $L_x:=\inf\{k\in \Z: x+ke_1\in \mc{C}_o\}$, and $L^*_x:=\inf\{k\in \Z: x+ke_1\in \mc{C}^*_o\}$.   Then trivially $L_x\ge L_x^*$ for every $x$.

For $z\in \Z^d$ we define $z_{\{1+\}}=\{z+ke_1:k\in \Z_+\}$, and for $A\subset \Z^d$
\begin{equation}
A_{\{1+\}}=\bigcup_{z \in A}z_{\{1+\}}.\label{A1+}
\end{equation}
The following result is a special case of a theorem from \cite{DREphase}, and shows that the ``left boundaries'' of $\mc{C}_o$ and $\mc{C}^*_o$ are the same.  Since $\mc{C}^*_o(p)$ is monotone decreasing in $p$, this gives a monotonicity result for the left boundary of $\mc{C}_o(p)$.
\begin{customthm}{$\boldsymbol{L}$}
\label{thm:other1}
Under the coupling \eqref{coupling}, for each  $x\in \Z^{d}$, and $p\in (0,1)$,  $L_x=L^*_x\in [-\infty,\infty)$ and $(\mc{C}_o(p))_{\{1+\}}=\mc{C}_o(p)^*$, a.s.
\end{customthm}
The following result (also a special case of a theorem from \cite{DREphase}) then reveals a phase transition for both models.
\begin{customthm}{$\boldsymbol{ C}$}
\label{thm:other2}
There exists $p_c\in (0,1)$ such that:
\begin{itemize}
\item[] if $p<p_c$ then $\mc{C}^*_o(p)=\Z^d$ almost surely, and 
\item[] if $p>p_c$ then $L_x^*(p)\in \R$ for every $x\in \Z^{d}$ almost surely (so $\mc{C}_o^*(p)\ne \Z^d$).
\end{itemize}
\end{customthm}
Our main result, Theorem \ref{shapetheorem} below, proves a shape theorem for $\mc{C}_o^*$ when $p$ is large.   In view of Theorem \ref{thm:other1} this immediately implies a shape theorem for the boundary of $\mc{C}_o$. 
For a set $H\subset \R^d$, and $s\ge 0$ we let $sH=\{sx:x\in H\}$.  Recall that a cone is any subset $C$ of $\R^d$ such that $sC=C$ for any $s>0$. We will prove that the region $n^{-1}\mc{C}_o^*$ for large $n$ is asymptotically a convex cone.  This cone will have an ``axis of symmetry'' $\vv:=\sum_{e\in\mc{E}_+}e$, so could be described as $\cup_{s\ge 0}s\chi$, where $\chi=C\cap\mc{S}_1$ is a convex subset of $\mc{S}_1:=\{x\in \R^d: x\cdot \vv=1\}$.  We call $\chi$ the \emph{shape} of $C$.

 We will let $\Pp$ denote the law of the half-orthant model with $p$ fixed.    
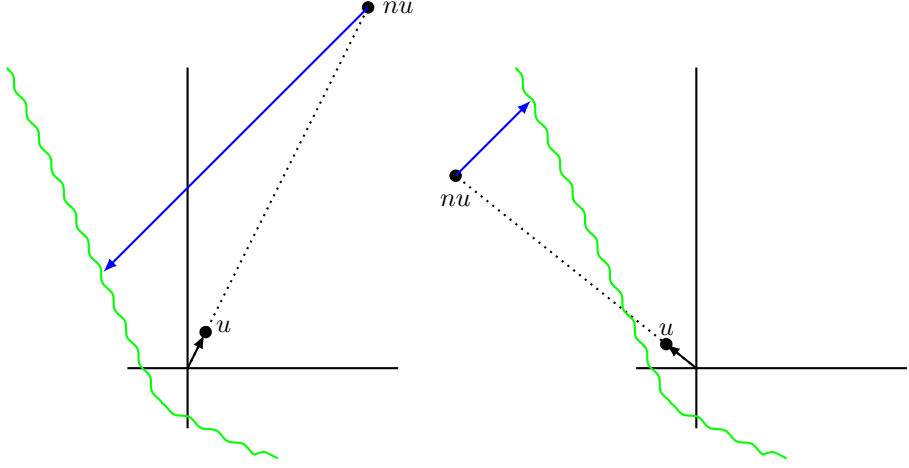
\begin{figure}
	\begin{center}
		\begin{tikzpicture}[thick,scale=.8]
		\path[draw=black] (0,-1)--(0,5);
		\path[draw=black] (-1,0)--(3.5,0);
		\path[->,draw=black,>=latex] (0,0)--(.28,.56);
		\node[circle,fill=black,scale=.5] at (.3,.6) {};
		\node at (.6,.7) {$u$};
		\node  at (3.5,6) {$nu$};
		\path[dotted,draw=black] (0,0)--(3,6);
		\node[circle,fill=black,scale=.5]  at (3,6) {};
		\path [draw=green,snake it]
		(-3,5) -- (-0.5,-0.5) --  (1.5,-1.5);
		\path[->,draw=blue,>=latex] (3,6)--(-1.4,1.6);
		\end{tikzpicture}
		\begin{tikzpicture}[thick,scale=.8]
		\path[draw=black] (0,-1)--(0,5);
		\path[draw=black] (-1,0)--(3.5,0);
		\path[->,draw=black,>=latex] (0,0)--(-.5,.4);
		\node[circle,fill=black,scale=.5] at (-.5,.4) {};
		\node at (-.5,.65) {$u$};
		\node  at (-4,2.8) {$nu$};
		\path[dotted,draw=black] (0,0)--(-4,3.2);
		\node[circle,fill=black,scale=.5]  at (-4,3.2) {};
		\path [draw=green,snake it]
		(-3,5) -- (-0.5,-0.5) --  (1.5,-1.5);
		\path[->,draw=blue,>=latex] (-4,3.2)--(-2.75,4.45);
		\end{tikzpicture}
		\caption{Illustrations of $\beta_n(u)$ in 2 dimensions for two choices of $u$.  The squiggly (green) line is the boundary of $\mc{C}_o$.  The (blue) line from $nu$ to the squiggly line is $\beta_n(u)\vv$.  In the first case $\beta_n(u)<0$, while in the second case $\beta_n(u)>0$.}
		\label{fig:beta}
	\end{center}
\end{figure}	
Let $\vv=\sum_{e\in\mc{E}_+}e$.  For $u\in\mathbb{Z}^d$ and $n \in \N$, let (see Figure \ref{fig:beta})
$$
\beta_n(u)=\inf\{k\in\mathbb{Z}: k\vv+nu\in\mc{C}_o\}.
$$

Let $O_r=\{x\in \Z^d:\|x\|_{\infty}\le r\}$. 
The following is our main result.
\begin{theorem}
\label{shapetheorem}
Fix $d\ge 2$.  For the half-orthant model, there is a $p_1<1$ such that the following hold for $p>p_1$.
\begin{enumerate} 
\item[\emph{(a)}] For $u\in \mathbb{Z}^{d}$ there is a deterministic $\gamma(u)$ such that  
$\frac{\beta_n(u)}{n}\to\gamma(u)$, $\Pp$-a.s.;
\item[\emph{(b)}] $\gamma(u+w)\le \gamma(u)+\gamma(w)$;  $\gamma(ru)=r\gamma(u)$; $\gamma(u+r\vv)=\gamma(u)-r$ for $r\in\mathbb{Z}_+$ and $u,w\in \mathbb{Z}^{d}$; $\gamma$ is symmetric under permutation of coordinates; $\gamma(u)\ge 0$ if $u\cdot\vv\le 0$; $\gamma(u)\le 0$ if $u$ lies in the positive orthant.
\item[\emph{(c)}] $\gamma$ extends to be a Lipschitz map $\mathbb{R}^{d}\to\mathbb{R}$ with these same properties but  for $r\in[0,\infty)$ and $u,w\in\mathbb{R}^d$. 
\item[\emph{(d)}] the set $C:=\{z\in\mathbb{R}^d: \gamma(z)\le 0\}$ is a closed convex cone, which is symmetric under permutations of the coordinates, contains the positive orthant, and is contained in the half-space $\{z: z\cdot\vv\ge 0\}$.  
\item[\emph{(e)}] $\frac1n\mc{C}_o\to C$ in the sense that for every $\epsilon>0$ and every $r<\infty$, the following holds $\Pp$-a.s. for sufficiently large (random) $n$:
$$
(O_r\cap \frac1n\mc{C}_o )\subset O_\epsilon+C
\quad\text{and}\quad
O_r\cap C\subset O_\epsilon+ \frac1n\mc{C}_o.
$$
\end{enumerate}
\end{theorem}
When $p=1$, the shape $\chi$ of $C$ is the simplex $\{x\in \R_+^d:x\cdot \vv=1\}$, and if also $d=3$ then $\chi$ is a (filled) triangle.  As we decrease $p$ the triangle becomes rounded.  See for example the third picture in Figure \ref{fig:3d}.
Note that (e) implies convergence in the pointed Gromov-Hausdorff metric. See \cite{ADH}. 

As remarked earlier, we will analyze $\beta_n(u)$ using subadditivity. The special role of the origin creates a lack of stationarity, so these subadditivity arguments are far from routine. Circumventing this obstacle is the main technical contribution of this paper. In doing so, we will rely on the exponential decay of certain probabilities, which is the main reason why we only have a proof for large $p$, rather than for all $p>p_c$. In other words, the following remains open:
\begin{open}
Prove that Theorem \ref{shapetheorem} holds for all $p>p_c$.
\end{open}

Note that \cite{Dur84} uses a block argument to go from large $p$ to all $p>p_c$, in the setting of oriented percolation in two dimensions. It would be interesting to find some analogue to this approach, in the context of Theorem \ref{shapetheorem}. 

According to \cite{DREphase} there are phase transitions for more general models of degenerate random environments.  We introduced $L_x$ above, which is analogous to $\beta_1(u)$ except that $e_1$ plays the role of $\vv$. Under Condition 2 of \cite{DREphase}, one could attempt to formulate shape theorems for such general models, in which case it may be more natural to use the $L_x$'s. We will therefore record the following variation of (a) of Theorem \ref{shapetheorem} (still in the setting of the half-orthant model). Let $Z$ denote the discrete hyperplane $\{y\in\Z^d: y\cdot e_1=0\}$. 
\begin{corollary}
\label{cor:Lxlimits}
Assume the conditions of Theorem \ref{shapetheorem}, and let $p>p_1$. For each $v\in Z$ there exists a deterministic $\zeta(v)\in \R$  such that 
	\begin{equation}
	n^{-1}L_{nv}\to \zeta(v), \quad \Pp-\text{almost surely as }n \to \infty.
	\label{eqn:openWshape}
	\end{equation}
\end{corollary}
\begin{open}
Prove a version of \eqref{eqn:openWshape} (or of Theorem \ref{shapetheorem}) for more general degenerate random environments e.g.  assuming Condition 2 of \cite{DREphase}.  
\end{open}

A comparison with the shape theorems of first passage percolation 
might suggest that the convergence in (e) of Theorem \ref{shapetheorem} (in which $\epsilon$ is fixed) fails to capture the fine structure of $\frac1n\mc{C}_o$. 
The following is intended to show that this is not actually the case. For $\delta>0$, let $C_\delta=\{z\in\R^d: \gamma(z)\le -\delta\}$. 
\begin{corollary}
\label{cor:finestructure} 
Assume the conditions of Theorem \ref{shapetheorem}, and let $p>p_1$. Then for every $\delta>0$ and every $R<\infty$, the following holds $\Pp$-a.s. for sufficiently large (random) $n$:
$$
O_R\cap C_\delta\cap\frac1n\mathbb{Z}^d\subset\frac1n\mc{C}_o.
$$
\end{corollary}
In other words, $\frac1n\mc{C}_o$ fills out the available lattice points of $C$, away from the boundary of $C$.

All of the above results concern the forward cluster $\mc{C}_o$.  A crucial difference between forward and backward clusters is that $\mc{B}_o$ can be finite for the orthant model.  For example, if $e_i\in \Omega_+$ and $-e_i\in \Omega_-$ for each $i\in [d]$ (this has positive probability for any $p \in (0,1)$) then there are no arrows pointing to the origin, so $\mc{B}_o=\{o\}$.  For the half-orthant model $\mc{B}^*_o$ will be infinite, since it contains $-\mathbb{Z}_+e_1$. 

For $x\in \Z^{d}$, let $R_x(p)=\sup\{k \in \Z:x+ke_1\in \mc{B}_o(p)\}$, and $R^*_x(p)=\sup\{k \in \Z:x+ke_1\in \mc{B}^*_o(p)\}$.  The following theorem is proved in \cite{DREphase}.
\begin{theorem}
\label{thm:B}
For the half-orthant model, let $p_c$ be as in Theorem \ref{thm:other2}.  Then
\begin{itemize}
\item[] if $p<p_c$ then $\mc{B}_o=\Z^d$ almost surely, and 
\item[] if $p>p_c$ then $R_x$ is finite for every $x\in \Z^{d}$.
\end{itemize}
\end{theorem}
The proof of Theorem \ref{shapetheorem} can be adapted to work for the cluster $\mc{B}_o$, with $\hat\beta_n(u)=\sup\{k\in\mathbb{Z}: k\vv+nu\in\mc{B}_o\}$.
\begin{theorem}
	\label{shapetheoremB}
	Fix $d\ge 2$.  For the half-orthant model, 
	with $p_1<1$ and $\gamma, C$ as in Theorem \ref{shapetheorem}, the following hold for $p>p_1$, $\Pp$-a.s.
	\begin{enumerate} 
		\item[\emph{(a)}] For $u\in \mathbb{Z}^{d}$, 
		$\frac{\hat{\beta}_n(u)}{n}\to\ -\gamma(-u)$;
		\item[\emph{(e)}] $\frac1n\mc{B}_o\to -C$ in the sense that for every $\epsilon>0$ and every $r<\infty$, the following holds for sufficiently large (random) $n$:
		$$
		(O_r\cap \frac1n\mc{B}_o )\subset O_\epsilon+-C
		\quad\text{and}\quad
		O_r\cap -C\subset O_\epsilon+ \frac1n\mc{B}_o.
		$$
	\end{enumerate}
\end{theorem}
Section \ref{sec:subadditive} is devoted to the proof of Theorem \ref{shapetheorem}. Section \ref{sec:lemmas} verifies the Lemmas required for that proof. Corollaries \ref{cor:Lxlimits} and \ref{cor:finestructure} are proved in Section \ref{sec:corollaries}.  
The proof of Theorem \ref{shapetheoremB} is omitted.

\section{Proof of Theorem \ref{shapetheorem}}
\label{sec:subadditive}
The results about limit cones in \cite{DRE2} used oriented percolation in $\Z^2$, and therefore subadditivity in an indirect way (see \cite{Dur84}). The higher dimensional analogue in Theorem \ref{shapetheorem} will directly rely on subadditivity, borrowing from the approach taken with the shape theorems of first passage percolation. See \cite{Kesten} or \cite{ADH}. In particular, our goal is to prove a shape theorem analogous to that of Cox and Durrett \cite{CD}.

Recall that we are working with the half-orthant model, where $\mu(\{\mc{E}_+\})=p=1-\mu(\{\mc{E}\})$.  We will prove Theorem \ref{shapetheorem} via a sequence of lemmas.

Recall that $\vv=e_1+\dots+e_d$. Our notation for coordinates will be that $x^{\sss[i]}=x\cdot e_i$ for $i\in [d]$.  For $0\le \eta\le 1$ consider the cone 
$$
K_\eta=\{x\in\mathbb{R}^d: x\cdot\vv\ge \eta\|x\|_1\}.
$$
The case $\eta=0$ is a half-space, while $\eta=1$ is the positive orthant. The following result (together with a simple application of the Borel-Cantelli Lemma) shows that for each $\eta\in [0,1)$, if $p$ is sufficiently large then $\mathcal{C}_o\subset K_\eta-M\vv$ for some random $M>0$ almost surely.
\begin{lemma}
	\label{SAWargument}
There exists $\theta(d)>1$ such that the following holds.  For $\eta\in [0,1)$ there is a $p_0(\eta,d)<1$ for which $p>p_0$ implies that there exists $c_{\ref{SAWargument}}=c_{\ref{SAWargument}}(\eta,d)>0$ such that $\Pp(\mathcal{C}_o\not\subset K_\eta-m\vv)\le c_{\ref{SAWargument}}\theta^{-md}$ for each $m \in \Z_+$. 
\end{lemma}
This result is a reformulation of Theorem 4.2 of \cite{DRE}. We will nevertheless give the full (but short) proof in Section \ref{sec:lemmas}, because a similar argument will be needed in other settings, elsewhere in this paper.

A similar result holds for the cluster $\mc{B}_o$, which then implies that the bi-connected cluster $\mc{M}_o=\mc{C}_o\cap\mc{B}_o$ is finite whenever $p>p_0(\eta,d)$ for some $\eta>0$. 
\begin{open}
\label{biconnectedconjecture}
For the orthant or half orthant model, show that $\mc{M}_o$ is a.s. finite, whenever $p>p_c$, where $p_c$ is as in Theorem \ref{thm:other2}.
\end{open}


Before we continue with our sequence of lemmas, note that for every $n \in \N$
\begin{equation}
\beta_n(u+r\vv)=\beta_n(u)-nr,\quad  \text{ for any $r \in \Z$.}\label{beta+1}
\end{equation}
 So if (a) of the Theorem holds for $u$, then it also holds for any $u+r\vv$ ($r\in \Z$) with  $\gamma(u+r\vv)=\gamma(u)-r$.   Since $o\in\mc{C}_o$, we know $\beta_1(o)\le 0$. By Lemma \ref{SAWargument}, $\beta_1(o)$ is an integrable random variable. By definition, $\beta_n(o)=\beta_1(o)$, so in fact 
\begin{equation}
\gamma(o)=\lim_{n\to\infty}\frac1n \beta_n(o)=0.\label{fishy}
\end{equation} 
Therefore part (a) of the theorem holds for $u=o$. By the above remark, it also holds for $u$ any multiple of $\vv$, with $\gamma(r\vv)=-r$.

Therefore, to prove part (a) of the theorem we may assume $u$ is not of the form $j\vv$.   Then there exists $v\in \R^d$ such that 
	\begin{equation}
	\label{wlogstuffforu}
	\text{$u\cdot v>0$ and $v\cdot\vv=0$,}
	\end{equation}
(for example, we may take $v$ to be the projection of $u$ onto the hyperplane orthogonal to $\vv$, i.e. $v=u-\frac{u\cdot \vv}{d}\vv$).  Fix any such $v$.  Then we may define $\sigma=\sigma(u,v)>0$ by
	\begin{equation}
	\label{sigmaconstant}
	\sigma=\frac{u\cdot v}{\|u\|_1\|v\|_\infty}.
	\end{equation} 

Define $\Lambda_{u,v}(m,n)=\{z\in \mathbb{Z}^{d}: m u\cdot v\le z\cdot v< n u\cdot v\}$. In other words, $\Lambda_{u,v}(m,n)$ is a slab in $\mathbb{Z}^{d}$, running orthogonal to $v$, and containing $mu$ and $nu$ on its boundary - see Figure \ref{fig:Lambdamn}.
\begin{figure}
	\begin{center}
		\begin{tikzpicture}[thick,scale=1]
		\path[->,draw=black,>=latex] (0,0)--(1,1);
		\node at (1,1.5) {$\vv$};
		\path[->,draw=black,>=latex] (0,0)--(1.5,0);
		\node at (1.6,0.2) {$u$};
		\path[->,draw=black,>=latex] (0,0)--(0.7,-0.7);
		\node at (0.8,-0.5) {$v$};
		\node[circle,fill=black,scale=.5] at (2.5,0) {};
		\node at (2.2,0.2) {$mu$};
		\node[circle,fill=black,scale=.5] at (3.5,0) {};
		\node at (3.8,-0.2) {$nu$};
		\path[<->,dotted,draw=black,>=latex] (-1,-1)--(2.4,2.4);
		\path[<->,draw=black,>=latex] (1.5,-1)--(4.9,2.4);
		\path[<->,dashed,draw=black,>=latex] (2.5,-1)--(5.9,2.4);
		\foreach \x in {-.8,-.6,-.4,-.2,0,.2,.4,.6,.8,1,1.2,1.4,1.6,1.8,2,2.2}
		{  \path[draw=gray] (2.5+\x,\x)--(3.5+\x,\x);  }
		\end{tikzpicture}
	\end{center}
	\caption{An illustration of $\Lambda_{u,v}(m,n)$ in two dimensions.  Here $\Lambda_{u,v}(m,n)$ is the hatched region, including the  solid line but not the  dashed line. $\Lambda_{u,v}(m)$ is the region from the dotted line to the solid line, including the former but not the latter.  $\Lambda_{u,v}(n,\infty)$ is the region to the right of  the dashed line.}
	\label{fig:Lambdamn}
\end{figure}
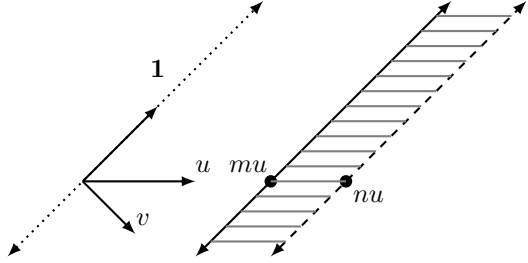
Note that if $x$ is any point of $\Lambda_{u,v}(m,n)$,
then $x+k\vv\in \Lambda_{u,v}(m,n)$ for every $k\in\Z$. We also define $\Lambda_{u,v}(-\infty,n)=\{z\in \mathbb{Z}^{d}: z\cdot v< n u\cdot v\}$ and $\Lambda_{u,v}(n,\infty)=\{z\in \mathbb{Z}^{d}: z\cdot v\ge n u\cdot v\}$.

For $\Lambda\subset\mathbb{Z}^{d}$ and $x\in\Lambda$, let $\mc{C}_x[\Lambda]$ be $x$ together with the set of $y\in \Z^d$ we can reach from $x$ by following arrows (consistent with the environment) that start in $\Lambda$. Note that $y$ itself need not be in $\Lambda$.
For $n\in \N$, set
$B_n(u,v)=\inf\{k: k\vv+nu\in\mc{C}_o[\Lambda_{u,v}(0,n)]\}$. This is $>-\infty$ by Lemma \ref{SAWargument} but (as remarked in the proof of Lemma \ref{momentbounds}) could $=+\infty$.

Recall the notation $p_0(\eta,d)$ from Lemma \ref{SAWargument}, and henceforth write $p_0(d)$ for $p_0(0,d)$.  Write $k^+=\max\{k,0\}$ and $k^-=\max\{-k,0\}$ for $k \in \R$.
\begin{lemma}
	\label{momentbounds}
	Assume \eqref{wlogstuffforu} and that $p>p_0(d)$. 
	There exist constants $\Gamma$ and $n_0$ such that
	\begin{align*}
	&\Ep[(B_n(u,v))^-]\le \Ep[(\beta_n(u))^-]\le \Gamma+\frac{n}{d}\|u\|_1
	\quad\text{for each $n\in \N$, and}\\
	&(\beta_n(u))^+\le (B_n(u,v))^+\le n\|u\|_1 
	\quad\text{for $n\ge n_0$.}
	\end{align*}
	Here $\Gamma$ depends only on $d$, but $n_0$ may depend on the choices of $u$ and $v$ (but not $p$).
	There exists an integrable random variable $Y\ge 0$ such that $\sup_n \frac{|\beta_n(u)|}{n}\le Y$. For $n\ge n_0(u,v)$, $B_n(u,v)$ has a moment generating function that is finite in a neighbourhood of 0.
\end{lemma}

Lemma \ref{momentbounds} will provide us with bounds that, together with a subadditivity argument, give the following result.
\begin{lemma}
	\label{lem:B_n}
For any $u\in \Z^d \setminus (\Z\vv)$, there is a constant $\gamma(u)\in \R$ such that for all $p>p_0(d)$, and all $v\in \R^d$ satisfying \eqref{wlogstuffforu}, $\P_p$-almost surely,
\begin{equation}
\gamma(u):=\lim_{n\to \infty}\frac{B_n(u,v)}{n}.
\end{equation}
Moreover, $n^{-1}B_n(u,v)\to \gamma(u)$ in $L^1$, and $\gamma(u)=\inf_{n\ge 1} n^{-1} \E_p[B_n(u,v)]$.
\end{lemma}

\medskip

We will need the following large-deviation type estimate.

\begin{lemma}
	\label{lemma:exponentialfirstbound} Assume that $u$ and $v$ satisfy \eqref{wlogstuffforu}. For any $p>p_0(d)$ and $\delta>0$, there is a $\bar c>0$ and an $n_1$ such that 
	$\Pp\Big(B_n(u,v)\ge n(\gamma(u)+4\delta)\Big)\le e^{-n\bar c}$ whenever $n\ge n_1$.
\end{lemma}

Since $\beta_n(u)\le B_n(u,v)$ it is clear that 
\begin{equation}
\limsup_{n\to\infty}\frac{\beta_n(u)}{n}\le\gamma(u).\label{betalimsup}
\end{equation}
To go further, we must address the possibility of paths that cross on either the positive or negative side of $\Lambda_{u,v}(0,n)$, and then backtrack, to get lower than $n\gamma(u)$.

 For $M,n\in \N$, let $A'_n(M)$ denote the event that there exists a self avoiding path consistent with the environment, running from $o$ to some point $k\vv+nu$ with $k<n\gamma(u)$, that hits $\Lambda_{u,v}(-\infty,-M)$. See Figure \ref{fig:stepfour}.   Similarly, let $A''_n(M)$ be the event that there exists a self avoiding path consistent with the environment, running from $o$ to some point $k\vv+nu$ with $k<n\gamma(u)$, that hits $\Lambda_{u,v}(M+n,\infty)$.  See Figure \ref{fig:stepfour}.   
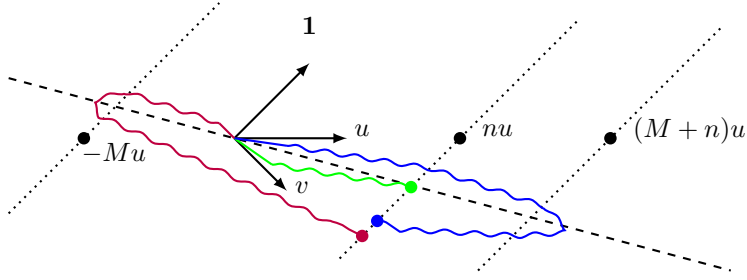
\begin{figure}
	\begin{center}
		\begin{tikzpicture}[thick,scale=1]
		\path[->,draw=black,>=latex] (0,0)--(1,1);
		\node at (1,1.5) {$\vv$};
		\path[->,draw=black,>=latex] (0,0)--(1.5,0);
		\node at (1.7,0.1) {$u$};
		\path[->,draw=black,>=latex] (0,0)--(0.7,-0.7);
		\node at (0.9,-0.65) {$v$};
		\node[circle,fill=black,scale=.5] at (-2,0) {};
		\node at (-1.6,-0.2) {$-Mu$};
		\node[circle,fill=black,scale=.5] at (3,0) {};
		\node at (3.5,0.1) {$nu$};
		\node[circle,fill=black,scale=.5] at (5,0) {};
		\node at (6.05,0.1) {$(M+n)u$};
		\path[dotted,draw=black,>=latex] (-3,-1)--(-0.2,1.8);
		\path[dotted,draw=black,>=latex] (1.4,-1.6)--(4.6,1.6);
		\path[dotted,draw=black,>=latex] (3.24,-1.76)--(6.5,1.5);
		\path[dashed,draw=black,>=latex] (-3,.8)--(6.6,-1.76);
		\path [draw=green] (0,0)--(.5,-.35);
		\path [draw=green,snake it] (.5,-.35) -- (2.35,-.65);
		\node[circle,fill=green,scale=.5] at (2.35,-.65) {};
		\path [draw=blue] (0,0)--(.7,-.1);
		\path [draw=blue,snake it] (.7,-.1) -- (2.6,-.4) -- (4.1,-.9) -- (4.35,-1.16) -- (3.8,-1.2) -- (1.9,-1.1);
		\node[circle,fill=blue,scale=.5] at (1.9,-1.1) {};
		\path [draw=purple] (0,0)--(-.2,.2);
		\path [draw=purple,snake it] (-.2,.2) -- (-1.4,.6) -- (-1.8,0.48) -- (-1.7,.3) -- (.7,-.8) -- (1.7,-1.3);
		\node[circle,fill=purple,scale=.5] at (1.7,-1.3) {};

		\end{tikzpicture}
	\end{center}
	\caption{A schematic illustrating the events $A'_n$ and $A''_n$. The dashed line lies in direction $u+\gamma(u)\vv$. The squiggly (green) curve in the middle represents a path from $o$ to $nu+B_n(u,v)\vv$. $A''_n(M)$ is the event that a path can be found such as the squiggly (blue) one on the right, that reaches significantly below the middle path, by first travelling into $\Lambda_{u,v}(M+n,\infty)$. $A'_n(M)$ is the event that a path such as the squiggly (purple) one on the left can be found, that also reaches well below the middle path, this time by first visiting $\Lambda_{u,v}(-\infty,-M)$.}
	\label{fig:stepfour}
\end{figure}

Finally let $\hat A_n$ be the event that there is a path from $o$ to some $k\vv$ with $k<0$ that is consistent with the environment and which reaches $\Lambda_{u,v}(n,\infty)$. See Figure \ref{fig:Anhat}.
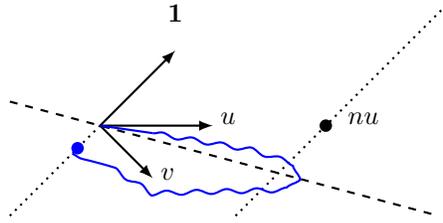
\begin{figure}
	\begin{center}
		\begin{tikzpicture}[thick,scale=1]
		\path[->,draw=black,>=latex] (0,0)--(1,1);
		\node at (1,1.5) {$\vv$};
		\path[->,draw=black,>=latex] (0,0)--(1.5,0);
		\node at (1.7,0.1) {$u$};
		\path[->,draw=black,>=latex] (0,0)--(0.7,-0.7);
		\node at (0.9,-0.65) {$v$};
		\node[circle,fill=black,scale=.5] at (3,0) {};
		\node at (3.5,0.1) {$nu$};
		\path[dotted,draw=black,>=latex] (1.8,-1.2)--(4.6,1.6);
		\path[dotted,draw=black,>=latex] (-1.2,-1.2)--(0,0);
		\path[dashed,draw=black,>=latex] (-1.2, .32)--(4.5,-1.2);
		\path [draw=blue] (0,0)--(.7,-.1);
		\path [draw=blue,snake it] (.7,-.1) -- (2,-.32) -- (2.4,-.45) -- (2.44,-.651) -- (2.3,-.8) -- (.4,-.85) -- (-.37,-.37);
		\node[circle,fill=blue,scale=.5] at (-.3,-.3) {};

		\end{tikzpicture}
	\end{center}
	\caption{A schematic illustrating the event $\hat A_n$. The dashed line lies in direction $u+\gamma(u)\vv$. $\hat A_n$ is the event that a path can be found such as the squiggly (blue) one, that reaches significantly below $o$, by first travelling into $\Lambda_{u,v}(n,\infty)$.}
	\label{fig:Anhat}
\end{figure}

\begin{lemma}
	\label{lemma:exponentialsecondbound} Assume that $u$ and $v$ satisfy \eqref{wlogstuffforu}. There exist  $c
	=c(u,v)>1$ and  $p_1=p_1(d) \in [p_0(d),1)$ such that
	$\Pp(\text{$A'_n(\lfloor cn\rfloor)$ i.o.})=\Pp(\text{$A''_n(\lfloor cn\rfloor)$ i.o.})=\Pp(\text{$\hat{A}_n$ i.o.})=0$ whenever $p>p_1$. 
\end{lemma}
Let $c>1$ be as in Lemma \ref{lemma:exponentialsecondbound}.
  Define for $n \in \N$
\begin{align}
\beta^0_n(u,v)&=\inf\{k: k\vv+nu\in\mc{C}_o[\Lambda_{u,v}(-\infty,n)]\},\label{beta0}\\
 \beta^1_n(u,v)&=\inf\{k: k\vv+nu\in\mc{C}_o[\Lambda_{u,v}(-\infty,\lfloor nc\rfloor+n)]\},\label{beta1}
\end{align}
so  $\beta_n(u)\le \beta^1_n(u,v)\le \beta^0_n(u,v)\le B_n(u,v)$, and therefore $\limsup\frac{\beta_n^1(u,v)}{n}\le \limsup\frac{\beta_n^0(u,v)}{n}\le\gamma(u)$. 
\begin{lemma}
	\label{lem:beta0}
For $p>p_1(d)$, $\frac{\beta_n^0(u,v)}{n}\to\gamma(u)$ $\P_p$-a.s. and in $L^1$.
\end{lemma}

\begin{lemma}
	\label{lem:beta1}
	For $p>p_1(d)$,
	$\frac{\beta_n^1(u,v)}{n}\to\gamma(u)$ a.s.
\end{lemma}

The above lemmas will be proved in Section \ref{sec:lemmas}.  Assuming the above lemmas, we are now ready to prove our 
main result.

\begin{proof}[Proof of Theorem \ref{shapetheorem}]
	$ $\\
\smallskip

\noindent{\bf (a):}  Let $\epsilon>0$ and $M_n=\lfloor nc \rfloor$, where $c$ is the constant of Lemma \ref{lemma:exponentialsecondbound}.  If $\frac{\beta_n(u)}{n}< \gamma(u)-2\epsilon$ then there is a self avoiding path from $o$ to some point $k\vv+nu$ with $k<n(\gamma(u)-2\epsilon)$. Such a path must either stay in $\Lambda_{u,v}(-\infty,M_n+n)$ or it must reach $\Lambda_{u,v}(M_n+n,\infty)$ (i.e.~$A_n''(M_n)$ occurs). By Lemmas \ref{lem:beta1} and  \ref{lemma:exponentialsecondbound} this is only possible for finitely many $n$. Therefore $\frac{\beta_n(u)}{n}\ge \gamma(u)-2\epsilon$ for all sufficiently large $n$ a.s.   Together with \eqref{betalimsup} this verifies (a), and in particular it shows that $\gamma(u)$ does not depend on the choice of $v$ satisfying \eqref{wlogstuffforu}.

\medskip

\noindent{\bf (b):}  Let $r\in \N$, and let $m=nr$.   From the definition of $\beta_n(u)$ we have that $\beta_{m/r}(ru)=\beta_m(u)$.  From this and part (a) it follows that 
\begin{equation}
\label{homothety}
\gamma(ru)=r\gamma(u)
\end{equation} 
for $r\in\mathbb{N}$.   This holds for $r=0$ as well, by \eqref{fishy}. The assertion that $\gamma(u+s\vv)=\gamma(u)-s$ for all $s\in \Z$ was verified in the paragraph prior to \eqref{fishy}. 

Symmetry of $\gamma$ under coordinate permutations is also straightforward, and the statement that $\gamma(u)\le 0$ when $u$ lies in the positive orthant follows from the fact that $\mc{C}_o$ contains this orthant.  Setting 
$\eta=0$ in Lemma \ref{SAWargument} shows that if $p>p_0(d)$ then there is an  $A\ge 0$ such that $\mc{C}_o\subset K_0-A\vv$. If $j\vv+nu\in K_0-A\vv$ then 
$[(A+j)\vv +nu]\cdot\vv\ge 0$ so $j\ge -A - \frac{n}{d}u\cdot \vv$. This proves that $\gamma(u)\ge -\frac{1}{d} u\cdot \vv$, so $\gamma(u)\ge 0$ whenever $u\cdot \vv\le 0$.

The proof that $\gamma$ is subadditive requires more care. We start by showing that $\gamma(u+w)\le \gamma(u)+\gamma(w)$  in the case that there is a $v\in\mathbb{R}^d$ such that $u\cdot v>0$, $w\cdot v>0$, and $v\cdot \vv=0$. Let 
\begin{equation}
\hat\Lambda_{u,v,w}(n)=\{z\in \mathbb{Z}^d: nu\cdot v\le z\cdot v< n(u+w)\cdot v\},\label{hatlambda}
\end{equation} 
and take (for $n \in \N$)
\begin{equation}
\hat B_n(u,v,w)=\inf\{k: (B_n(u,v)+k)\vv+n(u+w)\in\mc{C}_{B_n(u,v)\vv+nu}[\hat\Lambda_{u,v,w}(n)]\}
\end{equation}
Because $\Lambda_{u,v}(0,n)$ and $\hat\Lambda_{u,v,w}(n)$  do not overlap, we know that the environments in $\hat\Lambda_{u,v,w}(n)$ are independent of $B_n(u,v)$. Therefore $\hat B_n(u,v,w)$ has the same law as $B_n(w)$, and it follows from Lemma \ref{lem:B_n} that $\frac{1}{n}\hat B_n(u,v,w)$ converges in probability to $\gamma(w)$. Of course, $\frac{1}{n}B_n(u,v)$ and $\frac{1}{n}B_n(u+w,v)$ converge in probability to $\gamma(u)$ and $\gamma(u+w)$.  Now  $\hat\Lambda_{u,v,w}(n)\cup \Lambda_{u,v}(0,n) \subset \Lambda_{u+w,v}(0,n)$, so by concatenating paths we know that 
$$
B_n(u+w,v)\le B_n(u,v)+\hat B_n(u,v,w).
$$
Therefore $\gamma(u+w)\le \gamma(u)+\gamma(w)$. 

It remains to show subadditivity when such a $v$ does not exist. We start with the case $w=-u$, and show that 
\begin{equation}
\label{gammaofnegativeu}
\gamma(u)+\gamma(-u)\ge 0  \qquad (=\gamma(u+(-u)).
\end{equation}
Suppose that \eqref{gammaofnegativeu} fails.  Let $\epsilon>0$ be such that $\gamma(u)+\gamma(-u)+2\epsilon<0$, and $n_\epsilon\in \N$ be such that $n_\epsilon>-2/(\gamma(u)+\gamma(-u)+2\epsilon)$. Choose $v$ so $u\cdot v>0$ and $v\cdot\vv=0$.  
For $n\gg n_\epsilon$, set $k=\lceil n(\gamma(u)+\epsilon)\rceil$, $j=\lceil n(\gamma(-u)+\epsilon)\rceil$,  $z=nu+k\vv$, and $y=(j+k)\vv$. Then with high probability there is a path in $\mc{C}_o$ from $o$ to $z$, and a path in $\mc{C}_z$ from $z$ to $y$. Concatenating them gives a path from $o$ to $y$ that reaches $\Lambda_{u,v}(n,\infty)$. But $j+k\le n[\gamma(u)+\gamma(-u)+2\epsilon]+2<0$, and therefore $\Pp(\hat A_n)\to 1$, which contradicts Lemma 
\ref{lemma:exponentialsecondbound}. Thus  \eqref{gammaofnegativeu} holds. 

Now consider the general case that $u,w\in\mathbb{Z}^d$, yet no $v$ exists as above. Write $u=s\vv +u'$ and $w=t\vv + w'$ where $u'$ and $w'$ are $\perp \vv$. Then $s=\frac{1}{d}u\cdot\vv$ so $du'\in\mathbb{Z}^d$, and likewise $dw'\in\mathbb{Z}^d$. The half-spaces $\{z: z\cdot u'>0\}$ and $\{z: z\cdot w'>0\}$ do not intersect, as if they did, we could find a $v$ in their intersection, that is $\perp \vv$. Therefore in fact, $u'$ and $w'$ point in opposite directions. 
Thus there is an $x\in\mathbb{Z}^d$ and $i,j\in\mathbb{Z}$ such that $du'=ix$ and $dw'=jx$. 	By interchanging $u$ and $w$ if necessary, we may assume that $|i|\le |j|$.  Replacing $x$ by $-x$ if necessary, we may also assume that $j\ge 0\ge i$. Then using the fact that $i+j\in \Z_+$ we get 
\begin{align*}
\gamma(du+dw)&=\gamma(du'+dw')-d(s+t)\\
&=\gamma((i+j)x)-d(s+t)\\
&=(i+j)\gamma(x)-d(s+t)\\
&=i\gamma(x)-ds+j\gamma(x)-dt\\
&=-\gamma(-ix)-ds +\gamma(jx)-dt\\
&\le \gamma(ix)-ds+\gamma(jx)-dt\\
&=\gamma(du)+\gamma(dw).
\end{align*}	
Therefore $\gamma(u+w)\le \gamma(u)+\gamma(w)$ as required.  
\medskip

\noindent{\bf (c) and (d):}
It is now simple to extend the definition of $\gamma(u)$ to $u\in\mathbb{Q}^{d}$, by taking the limit over those $n$ for which $nu\in \mathbb{Z}^{d}$.  Then some algebra shows that (b) can be extended to $u,w\in \Q^d$ and $r \in \Q_+$.

From subadditivity, we see that
$$
|\gamma(w)-\gamma(u)|\le \max(|\gamma(u-w)|, |\gamma(w-u)|). 
$$
Therefore Lemma \ref{momentbounds} (which applies only to $\Z^d$) together with \eqref{homothety}  implies that  for $u,w\in \Q^d$,
\begin{equation}
|\gamma(u)-\gamma(w)|\le \|u-w\|_1.
\label{eqn:lipschitzcondition}
\end{equation}  
This in turn implies the existence of an extension of $\gamma$ to $\R^d$ that is uniformly Lipschitz 
(for $u \in \R^d$ take $\gamma(u)$ to be the limit of a sequence $\gamma(u_m)$ where $u_m$ are rationals approaching $u$).  
One can now verify that (b) holds with $u,w\in \R^d$ and $r\in \R_+$, by taking limits via sequences of rationals  $u_m,w_m,r_m$ converging to $u,w,r$ respectively.   This shows (c).  Now (d) also follows in turn, using the properties (b) now established on $\R^d$, $\R$:  $C$ is closed because $\gamma$ is continuous; it is a cone because if $x\in C$ and $r\ge 0$ then $\gamma(rx)=r\gamma(x)\le 0$, so $rx\in C$ (so $C$ is the cone with shape $\chi=\{x\in C:x\cdot \vv=1\}$); $C$ is convex because if $x,y\in C$ and $s\in (0,1)$ then $sx\in C$ and $(1-s)y \in C$ and $\gamma(sx+(1-s)y)\le \gamma(sx)+\gamma((1-s)y)\le 0$, so $sx+(1-s)y\in C$.  The remaining assertions of (d) are straightforward.
\medskip

	\noindent{\bf (e):}
Fix $r$ and $\epsilon$, and an integer $m>\frac{4r}{\epsilon}$.  The set $\mc{U}:=O_{m+2}$ is finite, 
so by part (a) we may find a random $J$ such that if $j\ge J$ then $|\frac{\beta_j(u)}{j}-\gamma(u)|<\frac12$ for every $u\in\mc{U}$. We will show the assertion of (e) for every $n$ such that 
\begin{equation}
\label{boundonn}
n\ge\max(\frac{Jm}{r}, \frac{4}{\epsilon}),
\end{equation}
so suppose $n$ satisfies \eqref{boundonn}. Set $k=\lceil \frac{nr}{m}\rceil\ge \frac{nr}{m}\ge J$.

We first show that $O_r\cap C \subset O_\epsilon+n^{-1}\mc{C}_o$.  
Let $z\in O_r\cap C$, so $\gamma(z)\le 0$. We know that if some $y\in \R^d$ satisfies $\|y\|_\infty\le m$ then there is a $u\in\mc{U}$ with $u^{\sss[i]}\le y^{\sss[i]}\le u^{\sss[i]}+1$ for every $i\in[d]$, and moreover that $u'=u+2\vv\in\mc{U}$.  Since $\|\frac{n}{k}z\|_\infty\le \frac{nr}{k}\le m$ we may choose $u$ and $u'$ as above, for $y=\frac{n}{k}z$. In particular,
$$
\|\frac{k}{n}u'-z\|_\infty
=\frac{k}{n}\|u'-\frac{n}{k}z\|_\infty\le\frac{2k}{n}
\le \frac{2}{n}(\frac{nr}{m}+1)<\epsilon
$$
by the choice of $m$ and the fact that $\frac{1}{n}<\frac{\epsilon}{4}$. So we need only show that $ku'\in\mc{C}_o$. 

Now $\gamma(z)\le 0$, so $\gamma(\frac{n}{k}z)\le 0$, and because $u+\vv-\frac{n}{k}z$ lies in the positive quadrant, also $\gamma(u+\vv-\frac{n}{k}z)\le 0$. By subadditivity it follows that $\gamma(u+\vv)\le 0$ and hence $\gamma(u')\le -1$. Because $k\ge J$, our choice of $J$ implies that $\beta_k(u')<0$, so in fact, $ku'\in\mc{C}_o$ as required.

Conversely, consider $z\in O_r\cap\frac{1}{n}\mc{C}_o$, so that by definition, there is a path consistent with the environment, from $o$ to $nz$. 
Construct $u,u'\in\mc{U}$, as above, so $u^{\sss[i]}\le (\frac{n}{k}z)^{\sss[i]}\le u^{\sss[i]}+1$ for each $i\in[d]$. Therefore $k(u+\vv)-nz$ belongs to the positive orthant, so there is also a path, consistent with the environment running from $nz$ to $k(u+\vv)$.  Concatenating the two paths, we see that $\beta_k(u+\vv)=\beta_1(k(u+\vv))\le 0$.  Recalling \eqref{beta+1}, it follows that 
$$
\frac{\beta_k(u')}{k}=\frac{\beta_k(u+\vv)}{k}-1\le -1.
$$
Since $k\ge J$, we obtain that $\gamma(u')\le -\frac12<0$. Therefore $u'\in C$ and in turn $\frac{k}{n}u'\in C$. As above, $\|\frac{k}{n}u'-z\|_\infty<\epsilon$, and the proof is complete.
\end{proof}

While the approach we have taken bears some relation to the subadditivity arguments of first passage percolation, it is in many ways is closer to the approach used for oriented percolation (see \cite{Dur84}). Our setting is more challenging, in part because there is no direction in which our model is fully oriented. For oriented percolation, the orientation immediately gives independence properties, which give rise to the stationarity required by the subadditive ergodic theorem. In our setting, paths can wander away and then return, and we must control the probabilities of that happening.

There is an alternative approach to our results that suggests itself, which is more closely aligned with first passage percolation. Namely to study $T(x,y)$ defined as the number of steps it takes, consistent with the environment, to travel from $x$ to $y$. Then $\mc{C}_o=\{y: T(o,y)<\infty\}$. Some aspects of the analysis become easier in this context (for example, subadditivity). But other problems arise, due to the fact that $T(x,y)$ can be infinite -- an issue also faced in \cite{Mourrat} and \cite{CerfTheret}. In our setting, the analysis of $T(x,y)$ appears even less tractable than in those works, because we have little control of $T(o,y)$ as $y$ approaches the boundary of $\mc{C}_o$. The approach taken in the current paper is designed specifically to deal with that boundary.

\section{Proofs of Lemmas}
\label{sec:lemmas}
\begin{proof}[Proof of Lemma \ref{SAWargument}.] We let the constant $c$ vary from line to line. There is a $\theta>1$ such that the number of self avoiding paths in $\mathbb{Z}^d$ from $o$ of length $n$ is at most $\theta^n$. For $m\in \Z_+$, let $\partial_{\eta,m}$ be the set of vertices $x$ adjacent to some vertex in $K_\eta-m\vv$ but not in $K_\eta-m\vv$. 
	
	If $\mathcal{C}_o\not\subset K_\eta-m\vv$ there is an $x\in\partial_{\eta,m}$  and a self avoiding path $\omega$ from $o$ to $x$ consistent with the environment. Set $x^+=\sum_{i=1}^d \max\{x\cdot e_i,0\}$ and $x^-=-\sum_{i=1}^d \min\{x\cdot e_i,0\}$. For some $k\ge 0$, $\omega$ takes $k+x^+$ steps in a direction from $\mc{E}_+$, and $k+x^-$ steps in a direction from $\mc{E}_-$. Therefore at least $k+x^-$ vertices come from $\Omega_-$.  If  $\theta^2(1-p)<1$ then we obtain
	\begin{align*}
	\Pp(\mathcal{C}_o\not\subset K_\eta-m\vv)
	&\le \sum_{x\in\partial_{\eta,m}}\sum_{k=0}^\infty \theta^{\|x\|_1+2k}(1-p)^{k+x^-}\\
	&\le c\sum_{x\in\partial_{\eta,m}}\theta^{\|x\|_1}(1-p)^{x^-}.
	\end{align*}

	For $x=y-m\vv$ we have $x\cdot \vv=y\cdot \vv -md$ and $\|y\|_1\le \|x\|_1+md$. If $x\notin K_\eta-m\vv$ then $y\notin K_\eta$ so 
	\begin{align*}
	x^+-x^-=x\cdot\vv=y\cdot\vv-md<\eta\|y\|_1-md& \le \eta\|x\|_1-md(1-\eta)\\
	& =\eta(x^++x^-)-md(1-\eta).
	\end{align*}
	Therefore $x^+<\frac{1+\eta}{1-\eta}x^- -md$ so $\|x\|_1=x^++x^-<\frac{2}{1-\eta}x^- -md$. 
	The number of such $x$ with $x^-=j$ is therefore at most $cj^d$,  uniformly in $m$. 
	Therefore 
	$$
	\Pp(\mathcal{C}_o\not\subset K_\eta-m\vv)
	\le c\theta^{-md}\sum_{j=0}^\infty[\theta^{\frac{2}{1-\eta}}(1-p)]^jj^d.
	$$
	Provided $p>1-\theta^{-\frac{2}{1-\eta}}$, this gives us a bound $c\theta^{-md}$, as claimed.
\end{proof}

\begin{proof}[Proof of Lemma \ref{momentbounds}.]
	For the first inequality, take $\eta=0$ in Lemma \ref{SAWargument} to see that if $p>p_0(d)$ then there is an integrable $A\ge 0$ such that $\mc{C}_o\subset K_0-A\vv$. Thus if $k=\inf\{j\in\Z: j\vv+nu\in K_0-A\vv\}$, it follows that $B_n(u,v)\ge\beta_n(u)\ge k$ and so $(B_n(u,v))^-\le(\beta_n(u))^-\le k^-$. But $j\vv+nu\in K_0-A\vv$ $\Leftrightarrow$ 
	$[(A+j)\vv +nu]\cdot\vv\ge 0$ $\Leftrightarrow$ $j\ge -A - \frac{n}{d}\sum u_i$. If this holds then $j\ge -A-\frac{n}{d}\|u\|_1$, so in particular, $k\ge  -A-\frac{n}{d}\|u\|_1$ and therefore 
	\begin{equation}
	\label{eqn:betaminus}
	(\beta_n(u))^-\le A+\frac{n}{d}\|u\|_1.
	\end{equation} 
	Therefore the desired inequality holds, with $\Gamma=\Ep[A]$. 
	
	We know that $\mc{C}_o$ contains the positive orthant, so $k\vv+nu\in\mc{C}_o$ whenever all its coordinates are non-negative. In other words, 
	\begin{equation}
	\label{eqn:betaplus}
	\beta_n(u)\le n\max_i\{(u^{\sss[i]})^-\}\le n\|u\|_1
	\end{equation}
	for every $n$. Combined with \eqref{eqn:betaminus}, this shows that the $\sup_n \frac{|\beta_n(u)|}{n}\le Y$, where $Y=A+(1+1/d)\|u\|_1$ is integrable.
	
	To complete the second set of inequalities of the Lemma, we use the same argument that gave us \eqref{eqn:betaplus}, but this time for $B_n(u,v)$. To make this work, we must show that if $x=k\vv+nu$ lies in the positive orthant, then it also lies in $\mc{C}_o[\Lambda_{u,v}(0,n)]$. We will be able to carry this out for $n\ge n_0$, where $n_0$ is chosen so that $n_0 u\cdot v>2\|v\|_\infty$. 
	\footnote{To show that such a restriction on $n$ is actually needed, suppose $d=3$, $v\cdot e_2<0$, $v\cdot e_3<0$, and $v\cdot e_1>-v\cdot e_2>0$. Take $u=-100\vv + e_1+e_2$, $k=100$, and $n=1$. Then $x=e_1+e_2$ and $u\cdot v=(e_1+e_2)\cdot v>0$. But if $o\in\Omega_+$ then we are unable to start a path that reaches $x$ via $\Lambda_{u,v}(0,1)$, since any initial step $e_i$ will exit $\Lambda_{u,v}(0,1)$.}
	
	We will construct a path $y_0,y_1,\dots,y_m$ from $y_0=o$ to $y_m=x$, where at each stage $y_{j+1}=y_j+e_{i_j}$ for some $i_j\in [d]$. The $i_j$'s will be chosen so that $0\le y_j\cdot v< nu\cdot v$ (for $j\neq m$), and  $0\le y_j^{\sss[i]}\le x^{\sss[i]}$ for each $i,j$. This shows that $x\in \mc{C}_o[\Lambda_{u,v}(0,n)]$. 
	
	Suppose we've got as far as $y_j$, and $y_j\neq x$. Consider $I=\{i\in [d]: y_j^{\sss[i]}< x^{\sss[i]}\}$, which must be non-empty (since $y_j\neq x$). If 
	\begin{equation}
	\|v\|_\infty\le y_j\cdot v< nu\cdot v-\|v\|_\infty, 
	\label{middleregioneqn}
	\end{equation}
	then choose any $i\in I$ as $i_j$. Since 
	$$
	0\le y_j\cdot v-\|v\|_\infty\le (y_j+e_i)\cdot v\le y_j\cdot v+\|v\|_\infty< nu\cdot v,
	$$
	the iteration continues. 
	
	If \eqref{middleregioneqn} fails then either $y_j\cdot v< \|v\|_\infty$ or $y_j\cdot v\ge nu\cdot v-\|v\|_\infty$. In the first case, choose $i_j$ to be any $i\in I$ with $e_i\cdot v>0$.  Note that there must be some such $i$, as otherwise 
	\begin{align*}
	x\cdot v&=\sum_{i \in I}x^{\sss[i]}v^{\sss[i]}+\sum_{i \notin I}y^{\sss[i]}_jv^{\sss[i]}\\
	&\le \sum_{i \in I}y^{\sss[i]}_jv^{\sss[i]}+\sum_{i \notin I}y^{\sss[i]}_jv^{\sss[i]}\\
	&= y_j\cdot v\le \|v\|_\infty,
	\end{align*}
	which contradicts the choice of $n_0$ (recall that $x\cdot v=nu\cdot v$).  Then $0\le y_j\cdot v<(y_j+e_i)\cdot v\le 2\|v\|_\infty<nu\cdot v$, because $n\ge n_0$, and again, the iteration continues. In the second case, choose $i\in I$ with $e_i\cdot v\le 0$, if possible. If there is indeed such an $i$, then $0<nu\cdot v-2\|v\|_\infty\le (y_j+e_i)\cdot v\le y_j\cdot v< nu\cdot v$ (again because $n\ge n_0$), and once more, the iteration continues. Finally, if in the second case all $i\in I$ have $e_i\cdot v> 0$ then choose any $i\in I$ as $i_j$. As long as $j<m$, $x$ is $y_j$ plus a sum of multiple such $e_i$, so we'll have $0\le y_j\cdot v< (y_j+e_i)\cdot v< x\cdot v= nu\cdot v$, and again, the iteration continues.
	
	From the above we have $|B_n(u,v)|\le A+n \|u\|_1$.  Finiteness of the moment generating function of $B_n(u,v)$ now follows, because Lemma \ref{SAWargument} establishes an exponential bound for the tail of the random variable $A$.
\end{proof}

\begin{proof}[Proof of Lemma \ref{lem:B_n}]
	For $n\in \N$, let $B_{0,n}(u,v)=B_n(u,v)$.  For $n>m>0$ let 
	\begin{equation}
B_{m,n}(u,v)=\inf\{k: (B_m(u,v)+k)\vv+nu\in\mc{C}_{B_m(u,v)\vv+mu}[\Lambda_{u,v}(m,n)]\}.\label{Bmn}
	\end{equation}
Then
	\begin{equation}
	B_n(u,v)\le B_m(u,v)+B_{m,n}(u,v),\label{Bsub}
	\end{equation}
	because concatenating the given paths from $o$ to $B_m(u,v)\vv+mu$ and from there to $(B_m(u,v)+B_{m,n}(u,v))\vv+nu$ produces a path from $o$ to the latter point, that also stays in $\Lambda_{u,v}(0,n)$. 
	
	We will apply the subadditive ergodic theorem (in a version due to Liggett \cite{Liggett}) to the $B_n(u,v)$. The stationarity properties required there follow easily from our construction, and an independence argument. Specifically, the fact that $\Lambda_{u,v}(0,m)$ and $\Lambda_{u,v}(m,n)$ are disjoint means that their environments are independent. Therefore if we shift the environments of $\Lambda_{u,v}(m,n)$ by a vector $Z\vv$, where $Z$ is determined by the environments of $\Lambda_{u,v}(0,m)$, then we get an environment whose law is the same, and which is still independent of the environments of $\Lambda_{u,v}(0,m)$. With $Z=B_m(u,v)$, it is this shifted environment that determines $B_{m,n}(u,v)$.  Therefore $B_{m,n}(u,v)$ is independent of $B_m(u,v)$.
	
	It follows that for every $k\ge 1$ and $m\ge 0$, $B_k$ and any $B_{m,m+k}$ have the same distribution. The remainder of hypotheses (1.8) and (1.9) of \cite{Liggett} follow similarly. Note that later on, in the proof of Lemma \ref{lem:beta0}], we will encounter a similar situation where the analogous condition holds only for $m\ge 1$, but not for $m=0$. This means that the result of \cite{Liggett} won't apply there, something we'll have to work around when we get to that point.
	
	It remains only to check the moment hypotheses in \cite{Liggett} which 
	are that  $\E[|B_n(u,v)|]<\infty$ for every $n$, and $\E[B_n(u,v)]\ge - cn$ for some $c$. In fact, it is easily seen that the argument there works if we assume instead that these inequalities hold for $n\ge n_0$, where $n_0$ is fixed. 
	The one point in Liggett's proof that isn't a trivial change comes five lines after (2.4a), where he uses that $\gamma_n\le k\gamma_m + \gamma_\ell$, where $\ell$ is small and $\gamma_\ell$ is finite. But we can get around that by using $\gamma_n\le (k-i)\gamma_m + \gamma_{im+\ell}$ for some $i$ chosen so $ \gamma_{im+\ell}<\infty$.
	The above inequalities then follow from Lemma \ref{momentbounds}. A similar remark applies to the argument in \cite{Liggett} which shows that $\gamma(u)=\inf_{n\ge 1} n^{-1} \E_p[B_n(u,v)]$.

	Applying the subadditive ergodic theorem, we have obtained that there is a deterministic $\gamma(u)\in(-\infty,\infty)$ such that $\frac{B_n(u,v)}{n}\to\gamma(u)$ a.s. and in $L^1$.  
\end{proof}

\begin{proof}[Proof of Lemma \ref{lemma:exponentialfirstbound}.]
	Since $\lim\frac{\E_p[B_n(u,v)]}{n}=\inf \frac{\E_p[B_n(u,v)]}{n}=\gamma(u)$ by Lemma \ref{lem:B_n}, we may choose $j_0$ (non-random) so large that $j_0\gamma(u)\le \E_p[B_{j_0}(u,v)]\le j_0(\gamma(u)+\delta)$. We may also ensure that $j_0\ge n_0$, where $n_0$ is as in Lemma \ref{momentbounds}. Take $n'_1$ sufficiently large that $n'_1\ge j_0$ and $n'_1\delta>2j_0\|u\|_1$. 
	
	For $n\ge n'_1$, write $n=j_0K+j$ where $j_0\le j<2j_0$. Set $Y_k=B_{(k-1)j_0,kj_0}(u,v)$ and $\bar y=\E_p[Y_1]$. Then
	$$
	B_n(u,v)\le \sum_{k=1}^K Y_k+B_{j_0K,n}(u,v)
	$$
	where the $Y_k$ are independent and identically distributed.
	Because $n-Kj_0\ge j_0\ge n_0$, Lemma \ref{momentbounds} implies that $B_{j_0K,n}(u,v)\le 2j_0\|u\|_1<n\delta$.   Therefore whenever $B_n(u,v)\ge n(\gamma(u)+4\delta)$, we must also have $\sum_{k=1}^KY_k\ge n(\gamma(u)+3\delta)$.   
	
	First suppose $\gamma(u)\ge 0$. Then $n(\gamma(u)+3\delta)\ge Kj_0(\gamma(u)+3\delta)\ge Kj_0(\gamma(u)+2\delta)\ge K(\bar y+j_0\delta)$. Now assume instead that $\gamma(u)<0$. Since $n\delta>2j_0\|u\|_1\ge -2j_0\gamma(u)$, and $n\le j_0(K+2)$, we have that $n(\gamma(u)+3\delta)\ge j_0(K+2)\gamma(u)+3n\delta\ge Kj_0(\gamma(u)+2\delta)+[n\delta+2j_0\gamma(u)]\ge K(\bar y+j_0\delta)$. Therefore in either case, whenever $B_n(u,v)\ge n(\gamma(u)+4\delta)$, we must also have $\sum_{k=1}^KY_k\ge K(\bar y+j_0\delta)$. 
	
	Lemma  \ref{momentbounds} shows that the moment generating function $\psi$ of the $Y_k$ is finite in a neighbourhood of 0.  Standard large deviations estimates then give the existence of a $t>0$ such that $\P_p(\sum_{k=1}^KY_k\ge K(\bar y+j_0\delta))<e^{-tK}$ for all $K$ sufficiently large (i.e.~whenever $n\ge n_1$, for some choice of $n_1\ge n'_1$ that makes $K\ge j_0$ as well as $K\ge 2$).

	
	Setting $\bar c=\frac{t}{2j_0}$ we use the bound $n\le j_0K + 2j_0= j_0(K+2)$ to conclude that $\bar c n\le\frac{ t(K+2)}{2}=tK(\frac12+\frac1K)\le tK$ (the latter since $K\ge 2$), and the conclusion of the Lemma holds.
\end{proof}

\begin{proof}[Proof of Lemma \ref{lemma:exponentialsecondbound}]
We will estimate $\Pp(A'_n(M))$ (where $M=\lfloor cn\rfloor$ for some $c>1$ to be chosen later) using the argument of Lemma \ref{SAWargument}.  See Figure \ref{fig:stepfour}. On the event $A'_n(M)$ we have a self-avoiding path $\omega$ (consistent with the environment) from $o$ to a point $x=j\vv+nu$ with $j<n\gamma(u)$ that reaches some point $y_0\in \Z^d$ with $y_0\cdot v\le -Mu\cdot v$.  Recall that $x^+=\sum_{i=1}^d \max\{x\cdot e_i,0\}$ and $x^-=-\sum_{i=1}^d \min\{x\cdot e_i,0\}$. The path takes $k+x^+$ steps using directions in $\mc{E}_+$ and $k+x^-$ steps using directions in $\mc{E}_-$, for some $k\ge 0$. For a given $k$ and $x$, the probability that such a path $\omega$ exists is at most $\theta^{\|x\|_1+2k}(1-p)^{k+x^-}$. 

Choose $0<\alpha<1$ and set $N=\lfloor n^\alpha\rfloor$. Set $\eta=0$ and assume $p>p_0(d)$. Lemma \ref{SAWargument} implies that there is an event $G$ of probability at most $c_{\ref{SAWargument}}\theta^{-Nd}$ such that off $G$ we have $x\in K_0-N\vv$. Therefore $x\cdot\vv \ge -Nd$, and so 
\begin{equation}
jd + n\|u\|_1\ge (j \vv+nu)\cdot\vv \ge -Nd.\label{elephant1}
\end{equation}
Let $n\ge n_0$, where $n_0$ is the constant in Lemma \ref{momentbounds}.  Then since $n^{-1}\beta_n(u)\le \|u\|_1$ for $n\ge n_0$ (by Lemma \ref{momentbounds}) we have that $j\le n\|u\|_1$. Choose $n'_0\ge n_0$ sufficiently large to make $Nd\le (d-1)n\|u\|_1$ when $n\ge n'_0$. Together with \eqref{elephant1} this gives
\begin{equation}
\label{boundonj}
\text{$|j|\le n\|u\|_1$ for $n\ge n'_0$.}
\end{equation}
It follows that for $n\ge n'_0$ we have 
\begin{equation}
\|x\|_1\le n\|u\|_1+|j|d\le n\|u\|_1(1+d).\label{xnorm}
\end{equation}  Therefore for fixed $k$ and $x$ as above, the probability that such a path $\omega$ exists is at most 
$\theta^{n\|u\|_1(1+d)+2k}(1-p)^k$.

It takes at least $\|y_0\|_1$ steps for the walk $\omega$ to reach $y_0\in H:=\{y\in \Z^d: y\cdot v\le -Mu\cdot v\}$.    But $\|y_0\|_1\cdot\|v\|_\infty\ge -y_0\cdot v\ge Mu\cdot v$. Therefore $\|y_0\|_1\ge M\sigma\|u\|_1$. From there $\omega$ must travel back and reach the half-space $H':=\{y': y'\cdot v\ge nu\cdot v\}$, and since $\|y'-y_0\|_1\|v\|_\infty\ge (y'-y_0)\cdot v\ge (n+M)u\cdot v$ for  $y'\in H'$, this takes at least $\sigma(n+M)\|u\|_1$ steps. It follows that the total length $2k+\|x\|_1$ of the path $\omega$ is at least $\sigma(n+2 M)\|u\|_1$.  Therefore $k\ge \frac{1}{2}(\sigma(n+2 M)\|u\|_1-\|x\|_1)$. For $k\ge n'_0$, \eqref{xnorm} implies that 
\begin{equation}
\label{lowerboundonk}
k\ge \frac{1}{2}\Big(\sigma(n+2 M)-n(1+d)\Big)\|u\|_1. 
\end{equation} 
Now choose $c>1$ so large that $\sigma(n+2\lfloor cn\rfloor )\ge n(2+d)$ for every $n\ge n'_0$, and let $M=\lfloor cn\rfloor $.  
By \eqref{lowerboundonk} (with $M=\lfloor cn\rfloor  $) we have $k\ge\frac{n\|u\|_1}{2}$ whenever $n\ge n'_0$. 

By \eqref{boundonj} there are at most $1+2n\|u\|_1$ possible choices for $j$ (and hence $x$), for any given $n$. Summing over $k$ and these $x$'s, and adding back the probability $\Pp(G)\le c_{\ref{SAWargument}}\theta^{-Nd}$, we obtain that
\begin{align*}
\Pp(A'_n(\lfloor cn\rfloor  ))&\le c_{\ref{SAWargument}}\theta^{-Nd} + (1+2n\|u\|_1)\theta^{n\|u\|_1(1+d)}\sum_{k\ge \frac{n\|u\|_1}{2}}[\theta^2(1-p)]^k\\
&\le c_{\ref{SAWargument}}\theta^{-Nd} + \frac{1+2n\|u\|_1}{1-\theta^2(1-p)}\theta^{n\|u\|_1(2+d)}(1-p)^{\frac{n\|u\|_1}{2}}
\end{align*}
for $n\ge n'_0$.  Recalling that $N=\lfloor n^\alpha\rfloor$, we see that these terms are summable for any $p>p_1$, provided $\theta^{2+d}(1-p_1)^{\frac12}<1$ and $p_1\ge p_0(d)$.  Note that  $p_1$ does not depend on $u$ and $v$, though $c$ may.  This proves 
 the result for $A'_n$.  The claims for $A''_n$ and $\hat{A}_n$ are proved similarly. See Figures \ref{fig:stepfour} and \ref{fig:Anhat}.
\end{proof}

\begin{proof}[Proof of Lemma \ref{lem:beta0}]
Let $p>p_0(d)$, and let $\epsilon>0$. Fix $c>1$ as in Lemma \ref{lemma:exponentialsecondbound}, and  let $M=\lfloor nc\rfloor$, $g=\lceil n(c\gamma(u)+2\epsilon)\rceil$, and $h=\lfloor n(\gamma(u)-3\epsilon)\rfloor$. Set $z=-g\vv-Mu$.  For $n \in \N$
	take 
	\begin{equation}	
	B^0_n=B^0_n(u,v)=\inf\{k: k\vv+nu\in\mc{C}_o[\Lambda_{u,v}(-M,n)]\}\label{B0n}
	\end{equation}
	 and 
	$$
	A_n=\{ B^0_n<n(\gamma(u)-3\epsilon)\}.
	$$
	We will show that
	\begin{align}
	\label{Anupperbound}
	\Pp(A_n)	&\le \Pp(o\notin \mc{C}_z[\Lambda_{u,v}(-M,n)])\\ 
	\label{Anupperbound2}
	&\qquad + \Pp(\exists\, k\vv+nu\in \mc{C}_z[\Lambda_{u,v}(-M,n)]\text{ with $k\le h$})\\ \nonumber
	&\to 0, \qquad \text{ as }n \to \infty.
	\end{align}
	See Figure \ref{fig:steptwo}. 
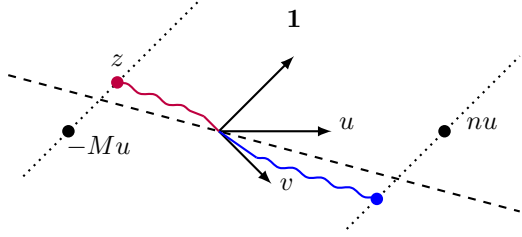
\begin{figure}
	\begin{center}
		\begin{tikzpicture}[thick,scale=1]
		\path[->,draw=black,>=latex] (0,0)--(1,1);
		\node at (1,1.5) {$\vv$};
		\path[->,draw=black,>=latex] (0,0)--(1.5,0);
		\node at (1.7,0.1) {$u$};
		\path[->,draw=black,>=latex] (0,0)--(0.7,-0.7);
		\node at (0.9,-0.7) {$v$};
		\node[circle,fill=black,scale=.5] at (-2,0) {};
		\node at (-1.6,-0.2) {$-Mu$};
		\node[circle,fill=black,scale=.5] at (3,0) {};
		\node at (3.5,0.1) {$nu$};
		\path[dotted,draw=black,>=latex] (-2.6,-.6)--(-0.2,1.8);
		\path[dotted,draw=black,>=latex] (1.7,-1.3)--(4,1);
		\path[dashed,draw=black,>=latex] (-2.8,.7467)--(4,-1.067);
		\path [draw=blue] (0,0)--(.5,-.35);
		\path [draw=blue,snake it] (.5,-.35) -- (2.1,-.9);
		\node[circle,fill=blue,scale=.5] at (2.1,-.9) {};
		\path [draw=purple] (0,0)--(-.2,.2);
		\path [draw=purple,snake it] (-1.35,.65) -- (-.2,.2);
		\node[circle,fill=purple,scale=.5] at (-1.35,.65) {};
		\node at (-1.35,0.95) {$z$};
		\end{tikzpicture}
	\end{center}
	\caption{A schematic illustrating the event $A_n$. The dashed line lies in direction $u+\gamma(u)\vv$. The squiggly (blue) curve on the right represents a path from $o$ to $nu+B_n^0\vv$ in $\Lambda_{u,v}(-M,n)$. $A_n$ is the event that this curve ends significantly below the dashed line, as shown in the figure.  The squiggly (purple) path on the left connects a point $z$ to $o$ through $\Lambda_{u,v}(-M,n)$ as in \eqref{Anupperbound}, where $z$ lies significantly above the dashed line. If both left and right paths exist, concatenating them yields a path as in \eqref{Anupperbound2}. }
		\label{fig:steptwo}
	\end{figure}
		To see the inequality, suppose that $A_n$ holds, so that there is a path from $o$ to some $k\vv+nu$ with $k\le h$ 
	that stays in $\Lambda_{u,v}(-M,n)$. Either there is no path from $z$ to $o$ 
	that stays in $\Lambda_{u,v}(-M,n)$, or there is such a path, in which case concatenating the two would give a path from $z$ to $k\vv+nu$.
	
	Next we must show that both terms \eqref{Anupperbound} and  \eqref{Anupperbound2} must $\to 0$. Consider \eqref{Anupperbound2}. By translation invariance it is equal to 
	\begin{align*}
	&\Pp(\exists\, k\vv+(M+n)u\in\mc{C}_o[\Lambda_{u,v}(0,M+n)]\text{ with $k\le g+h$})\\
	&=\Pp(g+h\ge B_{M+n}(u,v))\\
	&\le \Pp((M+n+1)\gamma(u) -n\epsilon +1\ge B_{M+n}(u,v))\\
	& \to 0, \quad \text{ as }n \to \infty,
	\end{align*}
	since 
	\[g+h\le n(c\gamma(u)+2\epsilon)+1+n(\gamma(u)-3\epsilon)\le (M+n+1)\gamma(u)-n\epsilon +1.\]

	Turning to \eqref{Anupperbound}, we have (as in the proof of Lemma \ref{momentbounds}) that if $n\ge n_0$, and $y=k\vv$ for some $k\le 0$, then $o\in\mc{C}_y[\Lambda_{u,v}(-M,n)]$. 
	Therefore translation invariance shows that for $n\ge n_0$,
	\begin{align*}
	&\Pp(o\in \mc{C}_z[\Lambda_{u,v}(-M,n)])
	\ge\Pp(\exists \,k\vv\in \mc{C}_z[\Lambda_{u,v}(-M,n)]\text{ with }k\le 0)\\
	&\qquad=\Pp(\exists\, k\vv+Mu\in \mc{C}_o[\Lambda_{u,v}(0,M+n)]\text{ with }k\le g)\\
	&\qquad\ge \Pp(\exists\, k\vv+Mu\in \mc{C}_o[\Lambda_{u,v}(0,M)]\text{ with }k\le g)\\
	&\qquad=\Pp(g\ge B_M(u,v)).
	\end{align*}
	Because $g \ge n(c\gamma(u)+\epsilon)\ge M\gamma(u)+n\epsilon$, this probability is at least 
	\[\Pp(M\gamma(u)+n\epsilon\ge B_M(u,v))\to 1, \quad \text{ as }n \to \infty.\]
	Therefore $\Pp(o\in \mc{C}_z[\Lambda_{u,v}(-M,n)])\to 1$.  This shows \eqref{Anupperbound}, and so $\Pp(A_n)\to 0$ as claimed. 
		
	We conclude from \eqref{Anupperbound}-\eqref{Anupperbound2} and Lemma \ref{lemma:exponentialsecondbound} that if $p>p_1$ then $\Pp(A_n\cup A'_n)\to 0$. Off this set, there can be no self avoiding path $\omega$ consistent with the environment, that stays in $\Lambda_{u,v}(-\infty,n)$, and which runs from $o$ to some point $k\vv+nu$ with $k<n(\gamma(u)-3\epsilon)$. For if $\omega$ stays in $\Lambda_{u,v}(-M,n)$ then $A_n$ would hold, and if it leaves $\Lambda_{u,v}(-M,n)$ then it enters $\Lambda_{u,v}(-\infty,-M)$ and $A'_n$ would hold. 
	
	We conclude that off $A_n\cup A'_n$ we have $\frac{\beta_n^0(u,v)}{n}\ge \gamma(u)-3\epsilon$. Combined with Lemma \ref{lem:B_n} and the fact that $\beta^0_n(u,v)\le B_n(u,v)$, this means that $\frac{\beta_n^0(u,v)}{n}\to \gamma(u)$ in probability.
	$L^1$ convergence now also follows because $\frac{\beta_n(u)}{n}\le \frac{\beta_n^0(u,v)}{n}\le\frac{B_n(u,v)}{n}$ and the latter $\to\gamma(u)$ in $L^1$, while  by Lemma \ref{lemma:exponentialsecondbound}, the former is bounded below by an integrable random variable.

	To extend the convergence in probability to almost surely convergence, we must carefully examine the proof of the version of the subadditive ergodic theorem given in \cite{Liggett}. 
	Define for $0< m<n$
	$$
	\tilde B_{m,n}(u,v)=\inf\{k: (\beta_m^0(u,v)+k)\vv+nu\in\mc{C}_{\beta^0_m(u,v)\vv+mu}[\Lambda_{u,v}(m,n)]\}.
	$$
	Examining \eqref{Bmn}, we see that $\tilde B_{m,n}(u,v)$ has the same distribution as $B_{m,n}(u,v)$.
         
         We also know that for $0< m<n$,
	$$
	\beta^0_n(u,v)\le \beta^0_m(u,v)+\tilde B_{m,n}(u,v). 
	$$
	The required moment bounds (for $n\ge n_0$) follow from Lemma \ref{momentbounds}, as in the proof of Lemma \ref{lem:B_n}.
	We know that the joint distribution of $(\tilde B_{m,m+k}(u,v); k\ge 1)$ does not depend on $m\ge 1$ (as in hypothesis (1.8) of \cite{Liggett}). We also know that for each $k\ge 1$, $(\tilde B_{nk,(n+1)k}(u,v); n\ge 1)$ is a stationary process (as in hypothesis (1.9) of \cite{Liggett}). What is not true is that $\beta^0_n(u,v)$ and $\tilde B_{m,m+n}(u,v)$ have the same law. 	 In other words hypothesis (1.8) of \cite{Liggett} does not apply in the present setting.
	
	To get around this, set $\gamma^0_n=\E_p[\beta^0_n(u,v)]$ and $\gamma_n=\E_p[B_n(u,v)]$.  From \eqref{Bsub} (and the fact that $\E_p[B_{m,n}(u,v)]=\E_p[B_{n-m}(u,v)]=\gamma_{n-m}$)  
	we know that $\gamma_n$ is a subadditive sequence, and therefore  $\lim\frac{\gamma_n}{n}=\inf\frac{\gamma_n}{n}=\gamma(u)$.  We also have $\gamma^0_n\le \gamma_n$, and $\gamma^0_{m+n}\le \gamma^0_m+\gamma_n$ for every $m$ and $n$.  Since also  $\lim\frac{\gamma_n^0}{n}=\gamma(u)$, we conclude that
	$$
	\gamma(u)=\lim_{n\to\infty}\frac{\gamma^0_n}{n}=\inf_{n\ge 1}\frac{\gamma_n}{n},
	$$
	which is a replacement for (2.1) of \cite{Liggett}. 
	
	It turns out that this is enough to make the rest of Liggett's proof work for us. In other words, once the above analogue to his (2.1) is shown, then the arguments for his (2.2), (2.3), and (2.4) all follow as written. From which we conclude that if $p>p_1$ then $\frac{\beta^0_n(u,v)}{n}\to\gamma(u)$ a.s.
	\end{proof}

\begin{proof}[Proof of Lemma \ref{lem:beta1}]
	Let $\epsilon>0$.  Let $c$ be as in Lemma \ref{lemma:exponentialsecondbound},  $\bar c$ be as in Lemma \ref{lemma:exponentialfirstbound}, and $c_{\ref{SAWargument}}=c_{\ref{SAWargument}}(0,d)$ be as in Lemma \ref{SAWargument}.  Take $\delta<\frac{\epsilon}{1+5c}$ and $M=\lfloor nc \rfloor$.  For $i\le n(\gamma(u)-\epsilon)$ let $w_i=i\vv+nu$. Define the event $C_i=\{\nexists\, k\le(M+n)(\gamma(u)-\delta)\text{ with }k\vv+(M+n)u\in\mc{C}_{w_i}[\Lambda_{u,v}(n,M+n)]\}$. By translation invariance, 
	\begin{align*}
	\P_p(C_i)&=\P_p\big(B_M(u,v)\ge (M+n)(\gamma(u)-\delta)-i\big)\\
	&\le \P_p(B_M
	(u,v)\ge M(\gamma(u)+4\delta))\le e^{-M\bar c},
	\end{align*}
	since 
	\begin{align*}
	(M+n)(\gamma(u)-\delta)-i
    &	\ge (M+n)(\gamma(u)-\delta)-n(\gamma(u)-\epsilon)\\
	&=M\gamma(u)-(M+n)\delta+n\epsilon\\
	&\ge M\gamma(u)+\delta(n(1+5c)-(M+n))\\
	&\ge M(\gamma(u)+4\delta).
	\end{align*}
	Define the event
	$$
	D_n=\{\beta^0_{M+n}(u,v)> (M+n)(\gamma(u)-\delta)\text{ but } \beta^1_n(u,v)\le n(\gamma(u)-\epsilon)\}.
	$$
	On $D_n$, there exists a path in $\mc{C}_o[\Lambda_{u,v}(-\infty, M+n)]$ from $o$ to some $i_0\vv+nu$ with $i_0\le n(\gamma(u)-\epsilon)$. Choose $0<\alpha<1$ and set $N=\lfloor n^\alpha\rfloor$. Then either $i_0< -N-\frac{n\|u\|_1}{d}$  or (by Lemmas  \ref{momentbounds} and \ref{lem:B_n}) $-N-\frac{n\|u\|_1}{d}\le i_0\le n(\gamma(u)-\epsilon)<n\gamma(u)\le n\|u\|_1$. 

	In the first case, $[(N+i_0)\vv+nu]\cdot\vv=(N+i_0)d+nu\cdot\vv<-n\|u\|_1+n\|u\|_1=0$, so $(N+i_0)\vv+nu\notin K_0$. In other words, $i_0\vv+nu\notin K_0-N\vv$, which by Lemma \ref{SAWargument} has probability at most $c_{\ref{SAWargument}}\theta^{-Nd}$.
	  
	There cannot be a path in $\Lambda_{u,v}(n,M+n)$ from $w_{i_0}$ to any $k\vv+(M+n)u$ with $k\le (M+n)(\gamma(u)-\delta)$, because concatenating the two paths would make 
	$\beta^0_{M+n}(u,v)\le (M+n)(\gamma(u)-\delta)$. Therefore 
	$$
	\Pp(D_n)\le c_{\ref{SAWargument}}\theta^{-Nd}+\sum_{i=-N-\frac{n\|u\|_1}{d}}^{n\|u\|_1}\Pp(C_i)\le c_{\ref{SAWargument}}\theta^{-Nd}+[1+N+n\|u\|_1(1+\frac1d)]e^{-M\bar c}.
	$$
	Summing over $n$ shows that $\sum_n \P_p(D_n)<\infty$. Lemma \ref{lem:beta0} tells us that if $p>p_1$ then $\beta^0_{M+n}(u,v)> (M+n)(\gamma(u)-\delta)$ eventually, and therefore $ \beta^1_n(u,v)\ge n(\gamma(u)-\epsilon)$ for all but finitely many $n$. In other words, $\frac{1}{n}\beta^1_n(u,v)\to \gamma(u)$ a.s. 
\end{proof}

\section{Proofs of Corollaries}
\label{sec:corollaries}

\begin{proof}[Proof of Corollary  \ref{cor:finestructure} ]
Fix $\delta>0$ and $R<\infty$. Set $\epsilon=\frac{\delta}{d}\land 1$ and $r=R+1$. Choose $n$ so large that the inclusions of (e) of Theorem \ref{shapetheorem} hold for $n$, with this $\epsilon$ and $r$. 

Let $z\in O_R\cap C_\delta\cap\frac1n\mathbb{Z}^d$. Then $\|\epsilon\vv\|_1\le\delta$, so by \eqref{eqn:lipschitzcondition} and the fact that $\gamma(z)\le -\delta$, we see that $\gamma(z')\le 0$, where $z'=z-\epsilon\vv$. Thus $z'\in C$. Since $z\in O_R$ and $\epsilon\le 1$, we know that $z'\in O_r$. So by choice of $n$, there is a $w\in \frac{1}{n}\mc{C}_o$ such that $\|w-z'\|_\infty<\epsilon$. Therefore $w'=\epsilon\vv - (w-z')$ lies in the positive orthant and we have $z=w+w'$. That implies that $o$ connects to $nw$ and $nw$ connects to $nz$, so $z\in\frac{1}{n}\mc{C}_o$, as claimed.
\end{proof}
We know that $\gamma(e_1)\le 0$. But it will be useful to know that this inequality is strict. 
\begin{lemma} Assume the conditions of Theorem \ref{shapetheorem}, and let $p>p_1$. Then $\gamma(e_1)<0$. 
\label{lem:e1strictnegativity}
\end{lemma}
\begin{proof}
Let $n,k>0$. Consider a path from $o$, constructed as follows. At sites in $\mc{E}_+$, follow $e_1$. At the first $n$ sites in $\mc{E}_-$, follow $-e_2$. At the next $n$ sites in $\mc{E}_-$, follow $-e_3$. Do the same in turn for $-e_4, \dots, -e_d$ till $n$ steps have been taken in the direction of each. After that, follow $e_1$ at sites in $\mc{E}_-$. 

This path will reach $n(ke_1-\vv)$ provided there are at least $n(d-1)$ sites in $\mc{E}_-$ among the first $n[(k-1)+d-1]$ sites visited. Choose $k$ so that $(1-p)[k+d-2]>d-1$. Then the law of large numbers implies that the above event has probability $\to 1$, when $n\to\infty$. 

In other words, $\beta_n(ke_1-\vv)\le 0$ with probability $\to 1$. Since $\frac1n\beta_n(ke_1-\vv)\to \gamma(ke_1-1)$ in probability, we get that $\gamma(ke_1-\vv)\le 0$. Therefore
$$
\gamma(e_1)=\frac1k\gamma(ke_1)=\frac1k[\gamma(ke_1-\vv)-1]\le-\frac1k<0.
$$
\end{proof}

\begin{proof}[Proof of Corollary \ref{cor:Lxlimits}  ]
Fix $w\in\Z^d$. Let $\zeta(w)=\inf\{t\in \R: \gamma(w+te_1)\le 0\}$.  If $s<t$ then 
$$
\gamma(w+te_1)\le \gamma(w+se_1)+\gamma((t-s)e_1)=\gamma(w+se_1)+(t-s)\gamma(e_1)<\gamma(w+se_1)
$$
by Lemma \ref{lem:e1strictnegativity}. So $\gamma$ is strictly decreasing and continuous, and therefore
\begin{align}
\gamma(w+te_1)\le 0&\Rightarrow t\ge \zeta(w)
\label{zetaineq1}\\\label{zetaineq2}
\gamma(w+te_1)\ge 0&\Rightarrow t\le \zeta(w).
\end{align}
Set $s=\liminf\frac{L_{nw}}{n}$ and $S=\limsup\frac{L_{nw}}{n}$ and let $M\in\Z$. We will show that $S\land M\le\zeta(w)\le s$, from which the desired result is immediate.

Consider the upper inequality first. Lemma \ref{SAWargument} implies that $s>-\infty$. There is nothing to show if $s=\infty$, so assume $s\in\R$. Then we may find a sequence $n_k$ such that $\frac{1}{n_k}L_{n_kw}\to s$. Therefore $n_kw+L_{n_kw}e_1\in\mc{C}_o$. Set $r=\|w\|_\infty + |s|+1<\infty$. Then for any $\epsilon>0$, (e) of Theorem \ref{shapetheorem} shows that for $k$ sufficiently large, 
$$
w+\frac{1}{n_k}L_{n_kw}e_1\in O_r\cap \frac{1}{n_k}\mc{C}_o\subset O_\epsilon+C.
$$
The latter is closed, so taking limits, we have $w+se_1\in O_\epsilon+C$ for every $\epsilon>0$. Since $C$ is closed, in fact $w+se_1\in C$, so $\gamma(w+se_1)\le 0$. By \eqref{zetaineq1} we get that $s\ge\zeta(w)$, as claimed.

To show the lower inequality, we argue similarly. If $S=-\infty$ there is nothing to show, so assume this is not the case. Let $\frac{1}{n_k}L_{n_kw}\to S$. Therefore $n_kw+\min(L_{n_kw}-1, n_kM)e_1\notin\mc{C}_o$. Set $R=\|w\|_\infty + |S\land M|+1<\infty$. Then for $k$ sufficiently large
$$
w+\frac{1}{n_k}\min(L_{n_kw}-1,n_kM)e_1\in [O_R\cap \frac{1}{n_k}\Z^d]\setminus \frac{1}{n_k}\mc{C}_o.
$$
If $\delta>0$, then for $k$ sufficiently large, Corollary \ref{cor:finestructure} implies that the above point $\notin C_\delta$. In other words, 
$\gamma(w+\frac{1}{n_k}\min(L_{n_kw}-1,n_kM)e_1)\ge -\delta$. Taking limits, we get that $\gamma(w+(S\land M)e_1)\ge -\delta$. Since $\delta>0$ was arbitrary, in fact $\gamma(w+(S\land M)e_1)\ge 0$.  By \eqref{zetaineq2} we obtain that $S\land M\le\zeta(w)$ as claimed, and we are done. \end{proof}

\bibliographystyle{plain}

\begin{thebibliography}{99}

\bibitem{ADH}
{A. Auffinger, M. Damron and J. Hanson, {\em 50 years of first-passage percolation}. University Lecture Series {\bf 68}, AMS, Providence R.I. (2017)}

\bibitem{Dur84}
{R. Durrett, ``Oriented percolation in two dimensions''. {\em Ann. Probab.} {\bf 12} (1984), pp. 999--1040}

\bibitem{CerfTheret}
{R. Cerf and M. Th\'eret, ``Weak shape theorem in first passage percolation with infinite passage times''. {\em Ann. Inst. Henri Poincar\'e Probab. Stat.} {\bf 52} (2016), pp. 1351--1381 }

\bibitem{CD}
{J.T. Cox and R. Durrett, ``Some limit theorems for percolation with necessary and sufficient conditions''. {\em Ann. Probab.} {\bf 9} (1981), pp. 583--603 }

\bibitem{DRE}
{M.~Holmes and T.S.~Salisbury,
``Degenerate random environments''.
{\em Random Structures Algorithms} {\bf 45} (2014), pp. 111--137}

\bibitem{DRE2}
{M. Holmes and T.S. Salisbury,   
``Forward clusters for degenerate random environments''. 
{\em Combin. Probab. Comput.} {\bf 25} (2016), pp. 744--765}

\bibitem{DREphase}
{M.~Holmes and T.S.~Salisbury,
	``Phase Transitions for Degenerate random environments''.  preprint (2019).}

\bibitem{RWDRE}
{M. Holmes and T.S. Salisbury, ``Random walks in degenerate random environments''. 
{\em Canad. J. Math.} {\bf 66} (2014), pp. 1050--1077}

\bibitem{RWDRE2}
{M. Holmes and T.S. Salisbury,  ``Conditions for ballisticity and invariance principle for 
random walk in non-elliptic random environment''.
{\em Electron. J. Probab.} {\bf 22} (2017), pp. 1--18}

\bibitem{Kesten}
{H. Kesten, ``Aspects of first passage percolation''. In {\em \'Ecole d'\'Et\'e de Probabilit\'es de Saint Flour XIV}, Lecture Notes in Mathematics {\bf 1180}, pp. 125--264. Springer-Verlag, Berlin (1986)}

\bibitem{Liggett}
{T.M. Liggett, ``An improved subadditive ergodic theorem''. {\em Ann. Probab.} {\bf 13} (1985), pp. 1279--1285}

\bibitem{Mourrat}
{J.-C. Mourrat, ``Lyapunov exponents, shape theorems and large deviations for the random walk in random potential''. {\em ALEA, Lat. Am. J. Probab. Math. Stat.} {\bf 9} (2012), pp. 165--209}

\bibitem{RAS}
{F. Fassoul-Agha and T. Sepp\"al\"ainen, {\em A course on large deviations with an introduction to Gibbs measures}. Graduate Studies in Mathematics {\bf 162}, AMS, Providence (2015)}

\bibitem{Zeit04}
{O. Zeitouni, ``Random walks in random environment''.
In {\em \'Ecole d'Et\'e de Probabilit\'es de Saint Flour}, Lecture
  Notes in Mathematics {\bf 1837}. Springer-Verlag, Berlin (2004)}

\end{thebibliography}

\end{document}